\documentclass[10pt]{article}
\usepackage{amsmath,amssymb,amsfonts,amsthm,mathtools,xfrac,enumitem,bbold}  
\usepackage{epstopdf}  
\usepackage{graphicx,float}
\usepackage[latin1]{inputenc} 
\usepackage[T1]{fontenc} 
\usepackage{hyperref}                 

\newcommand{\N}{\mathbb{N}}
\newcommand{\R}{\mathbb{R}}

\renewcommand{\{}{\lbrace}
\renewcommand{\}}{\rbrace}

\newcommand{\ind}{\text{ind}}

\def\hline
{
\vspace{-12pt}
\begin{center}
\rule{0.7\linewidth}{0.5pt}
\end{center} 
\vspace{12pt}
}

\newtheorem{thm}{Theorem}[section]
\newtheorem{thm*}{Theorem}
\newtheorem{proposition}[thm]{Proposition}
\newtheorem{lemma}[thm]{Lemma}
\newtheorem{remark}[thm]{Remark}
\newtheorem{corollary}[thm]{Corollary}
\newtheorem{definition}[thm]{Definition}

\title{
Prescribing Morse scalar curvatures: 
\\
incompatibility of non existence} 
 
\author
{Martin Mayer \\ 
\small Scuola Superiore Meridionale, Naples, ITALY \\
}  

\date{}

\begin{document}

\maketitle

\begin{abstract}
Given  a closed manifold of positive Yamabe invariant and for instance positive Morse functions upon it, 
the conformally prescribed scalar curvature problem raises the question, 
whether or not such functions can by conformally changing the metric be realised as the scalar curvature of this manifold. 
As we shall quantify, depending on the shape and structure of such functions, 
every lack of a solution for some candidate function 
leads to existence of energetically uniformly bounded solutions for entire classes of related candidate functions. 
\end{abstract}
 
\begin{center}
\end{center}

\tableofcontents   
\setcounter{tocdepth}{2}

\section{Introduction} 

 Give a closed and smooth Riemannian manifold $M$ of dimension $n\geq 2$, the prescribed scalar curvature problem is concerned with the question, which functions $K$ on $M$ can be prescribed as the scalar curvature $R=R_{g}$ on $M$ for some metric $g$ on $M$, i.e. solvability of $R_{g}=K$.
 
 If we allow to choose any metric, we wish, this problem has been solved in \cite{Kazdan_Warner_Gaussian},\cite{Kazdan_Warner_Scalar}. If we restrict the variation to a conformal class of metrics, i.e. we consider
$M=(M,g_{0})$ and are allowed to only choose among metrics $g$ conformal to $g_{0}$, i.e. $g\in [g_{0}]$ the conformal class of $g_{0}$, we refer to the conformally prescribed scalar curvature problem. 
Geometrically the restriction to a conformal class is natural for $n=2$ by the uniformization theorem and, since the Cotton tensor for $n=3$ and the Weyl tensor for $n\geq 4$ are invariant under a conformal change of the metric. 
Analogously to the Yamabe case of constant $K$ the conformally prescribed scalar curvature problem 
decays into three categories according to the sign of the conformal Yamabe invariant $Y$, cf. \eqref{Yamabe_Invariant}
We focus on $K,Y>0$, 
while we refer to \cite{MM7} for the case of a negative Yamabe invariant and to \cite{Escobar_Schoen_Deformation_Flatness} for $Y=0$
and the references therein.   
\subsubsection*{The Positive Case}

The equation to solve is of variational nature, i.e. solutions to the conformally prescribed scalar curvature problem $R_{g_{u}}=K>0$ correspond to critical points of an energy $J_{K}$, which in this case is positively lower bounded against the positive Yamabe invariant, cf. Section \ref{Section_The_Variational_Formulation}. Secondly the equation is critical with respect to Sobolev embeddings and therefore sublevel sets of $J_{K}$ are generally non compact.
To compensate this, one may consider symmetric or even radial situations to at least recover partial
compactness, 
cf. \cite{Catrina_Symmetric},\cite{Escobar_Schoen_Deformation_Flatness}, 
or pass to ordinary differential equations methods directly, cf. \cite{Chen_Infinite_Energy_Blow_Up}.
Generically however, but still considering $K>0$, we deal with

\subsubsection*{The  Morse case}

Let us briefly discuss a few relevant results, when $K>0$ is a generic Morse function
satisfying the classical, mild non degeneracy condition
$
\{ \nabla K=0 \}\cap \{ \Delta K=0 \} =\emptyset. 
$
\paragraph*{n=2} This is the case of $M\simeq S^{2}$.  
In \cite{Chang_Conformal_Deformations} we find an application of the Mountain Pass theorem under the assumption, 
that $K$ does have at least two maxima, while 
\begin{equation}\label{Laplacian_Positive_At_Non_Maxima}
\Delta K > 0
\; \text{ on } \;
\{ \nabla K=0 \} \setminus \{ K=\max K \}.   
\end{equation}
Connecting the latter maxima, which correspond to critical points at infinity of index 0, by a path, 
leads us to a criticality of index 1, 
which cannot stem from a critical point at infinity due to \eqref{Laplacian_Positive_At_Non_Maxima}.
Secondly, according to  \cite{Chang_S_2} and denoting by 
$$
p=\sharp\{ K=\max K \} 
\; \text{ and } \; 
q=\sharp\{ \nabla K=0 \; : \; m(K,\cdot)=1 \} 
$$ 
the number of maxima and saddle points respectively, we find existence, provided 
\begin{equation}\label{p_q_formula}
p \neq q+1.
\end{equation}
Condition \eqref{p_q_formula} may be expressed as an index counting formula, namely
\begin{equation}\label{Index_Counting_Formula_S_2}
-1 \neq \sum_{x\in \{ \nabla K =0 \} \cap \{ \Delta K < 0 \}}(-1)^{m(K,x)}.
\end{equation}

\paragraph*{n=3}
In case $M \simeq S^{3}$ the analogous results to $M\simeq S^{2}$ were 
proved in \cite{Bianchi_Min_Max} for the Mountain Pass argument and in \cite{Bahri_Coron_S_3}
for the index counting formula \eqref{Index_Counting_Formula_S_2}
\begin{equation}\label{index_countring_formula_n_3}
-1 \neq \sum_{x\in \{ \nabla K=0 \} \cap \{ \Delta K<0 \}  }(-1)^{m(K,x)},
\end{equation}
which like \eqref{Index_Counting_Formula_S_2} 
can be generalized to a total degree of the underlying equation, see \cite{Chang_S2_S3}, and
ensures solvability, cf. Section \ref{Sec_Obstruction_To_Non_Existence}. 
If $M\not \simeq S^{3}$, the problem becomes compact on low sublevels and solvability follows by direct minimization, 
cf. \cite{Escobar_Schoen_Deformation_Flatness}.

\paragraph*{n=4}
We refer to \cite{Bianchi_Min_Max} for a restrictive application of the Mountain Pass theorem and to
\cite{Ben_Ayed_M_4}
for a delicate analysis of the lack of compactness and a resulting index counting formula.
\paragraph{n $\mathbf{\geq} $ 5}
Mountain pass results are available under restrictive assumptions, cf. \cite{Bianchi_Min_Max},\cite{MM4},
while the corresponding index counting formula is never violated, cf. Lemma \ref{lem_index_counting_trivial},
and hence does not yield a solution.
In case of Einstein manifolds and under a pinching assumption 
existence is assured, provided $K$ admits more than one critical point of negative Laplacian. 
There is also an invariant besides the Euler characteristic, which assures existence, 
cf. \cite{Bahri_Invariant},\cite{Yacoub_Invariant_High_Dimensions}, but, as far as we are aware of, 
its computation is difficult.

\subsubsection*{Ambiguity}

Generally and in particular for positive Morse functions, as we shall discuss in Section \ref{Sec_Obstruction_To_Non_Existence}, 
these index counting formulae may fail to yield the existence of solutions and 
in that case we say, that the total degree is zero or trivial, which in our setting is actually always the case, when $n\geq 5$
and, depending on the complexity of the Morse structure of the functions under consideration, leaves many scenaria open, 
when $2\leq n \leq 4$. And at the same time mountain pass or other min-max principles may fail too. 
Here we deal with such situations.

\

To illustrate our main result, Theorem \ref{thm1}, and in order to have a practical example at hand, 
consider the smooth heart $S^{2} \simeq S \subset S^3$, cf.  Figure \ref{fig_smooth_heart} and Section \ref{section_connecting_orbits} for further discussions.

\begin{figure}[H]
\centering
\includegraphics[scale=1]{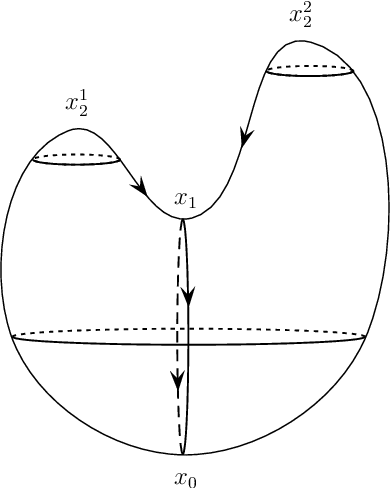}
\caption{The smooth heart as a height function}
\label{fig_smooth_heart}     
\end{figure} 

\noindent
Precisely for a function $K>0$ on $S^{3}$ we imagine for $K^{-1}=\frac{1}{K}$ exactly six critical points, which are
\begin{enumerate}[label=(\roman*)]
\item two local maxima $x_{3}^{1},x_{3}^{2}\in S^{3}\setminus S$, one inside and one outside $S$  of index 
$$m(K^{-1},x_{3}^{1})=m(K^{-1},x_{3}^{2})=3.$$
They are \textbf{not} depicted in Figure \ref{fig_smooth_heart}.  

\item two critical points $x_{2}^{1},x_{2}^{2}\in S $ of index 
$
m(K^{-1},x_{2}^{2})
=
m(K^{-1},x_{2}^{2})=2
$.
\item one critical point $x_{1}\in S$ with index 
$m(K^{-1},x_{1})=1$
\item one unique, global minimum $x_{0}\in S$ of index 
$m(K^{-1},x_{0})=0$. 
\end{enumerate}

\noindent
The height in Figure \ref{fig_smooth_heart} is to visualise their relative values on $S\subset S^{3}$, i.e.
\begin{equation}\label{height_relations}
K^{-1}(x_{0})<K^{-1}(x_{1})<K^{-1}(x_{2}^{1})\leq K^{-1}(x_{2}^{2}).
\end{equation}
We first show conformal solvability of $R=K$ on $S^{3}$ in a case, 
where \eqref{index_countring_formula_n_3} or equivalently
\eqref{index_counting_formula} 
do not. 
\begin{thm}\label{thm2} 
For the smooth heart $S$ on $S^{3}$ suppose, that 
\begin{enumerate}[label=(\roman*)]
\item 	$ K^{-1}(x_{2}^{1})\leq K^{-1}(x_{2}^{2}) $
\item   $ \Delta K(x_{2}^{1}) > 0 > \Delta K(x_{2}^{2}),\Delta K(x_{1})$.
\end{enumerate}
Then there is a solution $y\in \{\partial J_{K} = 0 \}$ with
$
J_{K}(y)\leq J_{K}(\delta_{x_{2}^{2}})
$. 
\end{thm}	
For the \hyperref[proof_thm_2]{proof}, which is based on min-max principle,   
see Section \ref{section_connecting_orbits}. 
Due to the index counting formula \eqref{index_countring_formula_n_3}  and by Theorem \ref{thm2} we then
can solve the corresponding conformally prescribed scalar curvature problem, unless
$$\Delta K(x_{0})<0<\Delta K(x_{2}^{2})$$ 
for the minimum $x_{0}$ of $K^{-1}$ and highest value saddle point $x_{2}^{2}$ of $K^{-1}$ and 
\begin{equation}\label{scenaria}
\; \text{ either } \;
\Delta K(x_{1}),\Delta K(x_{2}^{1})>0
\; \text{  or } \;
\Delta K(x_{1}),\Delta K(x_{2}^{1})<0. 
\end{equation}
These two scenaria are particularly hard,   
since in both cases of \eqref{scenaria} the total degree is zero and the problem is \textit{stable} at infinity, 
as we shall discuss in Section \ref{Sec_Obstruction_To_Non_Existence}, and hence
there are no obviously applicable min-max principles at our disposition. 

Note, that these two scenaria, i.e. the cases of \eqref{scenaria}, 
are essentially distinguished by the sign of the Laplacian at non extremal critical points. 
Interestingly - and perhaps surprisingly - Theorem \ref{thm1} or rather its $3$-dimensional version tells us, 
that \textbf{at most} one of these two cases is not solvable. 
And unfortunately we do not know, which one, since the proof of Theorem \ref{thm1} is by contradiction, thus non constructive. 

\
 
To formalize the notion of Morse functions, which are essentially distinguished by the sign of the Laplacian at their critical points, 
let us start with the following construction. Consider sequences 
$$\underline{k},\overline{k}=(\underline{k}_{l})_{l\in \N},(\overline{k}_{l})_{l\in \N}\subset \R_{+}$$ 
with 
\begin{equation}\label{ordering_the_k_i}
0
<
\underline{k}_{1} \leq \overline{k}_{1}
<
\underline{k}_{2}\leq \overline{k}_{2}
< 
\underline{k}_{3}\leq \overline{k}_{3}
<
\ldots
<\infty
\end{equation}
\begin{definition}\label{def_M_S}
Let  
$\mathcal{M}^{+,S,\underline{k},\overline{k}}_{i_{1},\ldots,i_{m} }(M)
=
\{  
0<K\in C^{\infty}(M)
\;:\; K \text{ satisfies } \text{(i)-(iii)}
\}
$, 
where
\begin{enumerate}[label=(\roman*)]
\item $\{\nabla K=0\}\cap \{\Delta K=0\}=\emptyset$.
\item $K$ is a positive Morse function with Morse structure $\mathcal{S}$ and 
$$\sharp(\{\nabla K=0\}\setminus \{ m(K,\cdot)=0  \}  ) = m.$$
\item 
With some \textbf{fixed} indices 
$i_{1}<\ldots< i_{m}\in \N$
we may label 
$$
\{ \bar{x}_{i_{1}},\ldots,\bar{x}_{i_{m}} \}= \{\nabla K=0\}\setminus \{ m(K,\cdot)=0  \}
$$
such, that for $j=1,\ldots,m$
\begin{equation}\label{energy_stip_single_points}
\underline{k}_{i_{j}}\leq \frac{1}{K^{\frac{n-2}{n}}(\bar{x}_{i_{j}})}\leq \overline{k}_{i_{j}}.
\end{equation}
\end{enumerate}
We say, that elements of $\mathcal{M}^{+,S,\underline{k},\overline{k}}_{i_{1},\ldots,i_{m} }(M)$
are  \textbf{simply spreading} and we call a subset
\begin{equation*}
\mathcal{M}
\subseteq
\mathcal{M}^{+,S,\underline{k},\overline{k}}_{i_{1},\ldots,i_{m} }(M)
\end{equation*}
a \textbf{spread} in $\mathcal{M}^{+,S,\underline{k},\overline{k}}_{i_{1},\ldots,i_{m} }(M)$, if
\begin{enumerate}
\item[(iv)]	for all $A\subseteq \{ 1,\ldots,m \} $ there exists  $i\in \N$ such, that for all $K\in \mathcal{M}$
\begin{equation}\label{energy_strip_multi_bubbles}
\underline{k}_{i} \leq c_{n}(\sum_{j\in A}\frac{1}{K^{\frac{n-2}{2}}(\bar{x}_{i_{j}})})^{\frac{2}{n}}\leq \overline{k}_{i},
\end{equation} 
and we have \textbf{uniqueness} in the sense, that, if for $A_{1},A_{2} \subseteq \{ 1,\ldots,m \}$ and $K_{1},K_{2}\in \mathcal{M}$
\begin{equation}\label{uniqueness_condition}
\underline{k}_{i}
\leq 
c_{n}(\sum_{j\in A_{1}}\frac{1}{K^{\frac{n-2}{2}}(\bar{x}_{i_{j}})})^{\frac{2}{n}}
,
c_{n}(\sum_{j\in A_{2}}\frac{1}{K^{\frac{n-2}{2}}(\bar{x}_{i_{j}})})^{\frac{2}{n}}
\leq 
\overline{k}_{i},
\end{equation}
then $A_{1}=A_{2}$. 
\end{enumerate}
Finally we say, that an element
$K\in \mathcal{M}^{+,S,\underline{k},\overline{k}}_{i_{1},\ldots,i_{m} }(M)$ is \textbf{spreading}, 
if 
$\{ K \} \subset \mathcal{M}^{+,S,\underline{k},\overline{k}}_{i_{1},\ldots,i_{m} }(M)$
is a spread. 
\end{definition}	
\begin{remark}
To avoid confusion and recalling, that $i_{1} < \ldots < i_{m}$ are fixed, we emphasize, that by (iii)  
$$
\; \forall \; K\in \mathcal{M}^{+,S,\underline{k},\overline{k}}_{i_{1},\ldots,i_{m} }(M)
\; : \; 
\underline{k}_{i_{j}}\leq \frac{1}{K^{\frac{n-2}{n}}(\bar{x}_{i_{j}})}\leq \overline{k}_{i_{j}}.
$$
Also note, that, as the number of critical points is limited by the Morse structure, 
we may assume $$\sup_{i\in \N} \overline{k}_{i}<\infty.$$ 
Some comments are in order to clarify (i)-(iv) above.
\begin{enumerate}[label=(\roman*)]
\item is the usual non degeneracy assumption, which for Morse functions allows us to identify the critical points at infinity easily, cf. \cite{MM3}.
\item Given a Morse function $K$ on $M$, consider the space $L_{K}$ of Morse functions on $M$, whose critical points correspond one to one to those of $K$ with coinciding Morse indices. Then a Morse structure on $M$ is just any space $L_{K}$ induced by some Morse function $K$.	 
\item tells us, that irrespective of a particular choice of $K$ the critical points, which potentially induce a single bubbling critical point at infinity by having a negative Laplacian, are distinguished by being contained in different energy strips, cf. \eqref{energy_stip_single_points} and \eqref{critical_energies}.
\item is to say first, that for any  $K\in \mathcal{M}$ all potential critical points at infinity of $K$
are contained in some energy strips, which do not depend on the function $K$, but on the spread $\mathcal{M}$. 
And conversely any such energy strip determines for every $K\in \mathcal{M}$ the \textbf{unique}  
combination of critical points of $K$, which may produce a critical point at infinity within this energy strip, 
cf. \eqref{critical_energies}. 
\end{enumerate}
Spreading functions are dense in any $C^{k}(M,\R_{>0})$. Moreover a sufficiently small $C^{0}$-neighbourhood of a spreading function 
$K$ contains an abundance of spreading functions with varying signs of the Laplacian at non extremal critical points, 
cf. Appendix \ref{sec_sign_of_laplacian}. In particular any spreading function 
$K$ creates locally, i.e. close to it, some spread $\mathcal{M}$ with $K\in \mathcal{M}$, 
on which Theorem \ref{thm1} below is applicable.
\end{remark}

We then associate to each simply spreading function its critical points with negative Laplacian and 
thereby obtain from \eqref{energy_stip_single_points} a map 
\begin{equation}\label{the_map_N}
\begin{split}
\mathcal{N}
: \;&
\mathcal{M}^{+,S,\underline{k},\overline{k}}_{i_{1},\ldots,i_{m} }(M)
\longrightarrow 
\mathbb{P}(m)=\{ A \subseteq \{ 1,\ldots,m \}  \} \\
: \; &
K\longrightarrow \{ \bar{x}_{i_{j}}\; : \; j\in A\subset \{ 1,\ldots,m \}  \} = \{ \nabla K=0 \} \cap \{ \Delta K < 0 \}  
\longrightarrow 
A,
\end{split}
\end{equation}
which induces a partition of $\mathcal{M}^{+,S,\underline{k},\overline{k}}_{i_{1},\ldots,i_{m} }(M)$ by
identifying those functions, whose critical points with negative Laplacian correspond one to one by having values in the same energy strip. 
This justifies
\begin{definition}\label{def_partition}
Let
\begin{equation*} 
\tilde{\mathcal{M}}^{+,S,\underline{k},\overline{k}}_{i_{1},\ldots,i_{m} }(M)
=
\mathcal{M}^{+,S,\underline{k},\overline{k}}_{i_{1},\ldots,i_{m} }(M)/\sim \;\; \text{ via }\;\;
\; K_{1}\sim K_{2} \Longleftrightarrow
\mathcal{N}(K_{1})
=
\mathcal{N}(K_{2})
\end{equation*}  
and denote by 
$\tilde{\mathcal{M}}$ 
for a spread
$\mathcal{M}\subseteq \mathcal{M}^{+,S,\underline{k},\overline{k}}_{i_{1},\ldots,i_{m} }(M)$
the induced quotient.
\end{definition}
With these notions at hand we can state our main result, which shows, 
that  in a certain \textit{local} sense there exists always at most one combination of critical points with negative Laplacian such, 
that we cannot assure existence of solutions - even among situations with vanishing total degree, 
in particular when the index counting formula
\eqref{index_counting_formula}, of which \eqref{p_q_formula} and \eqref{index_countring_formula_n_3} are special cases, fails. 
\begin{thm}\label{thm1}
Let  
$\mathcal{M}
\subseteq
\mathcal{M}^{+,S,\underline{k},\overline{k}}_{i_{1},\ldots,i_{m} }(M)
$
be a spread. 
Then, if
\begin{equation*}
\frac{\overline{\kappa}}{\underline{\kappa}}
<
\inf_{i\in \N}\frac{\underline{k}_{i} }{\overline{k}_{i-1}}
\; \text{ and } \; 
\; \forall \; K_{1},K_{2}\in \mathcal{M}
\; : \; 
\frac{K_{1}}{\overline{\kappa}^{\frac{n}{n-2}}}
<
K_{2} 
<
\frac{K_{1}}{\underline{\kappa} ^{\frac{n}{n-2}}}
\end{equation*}
for some 
$0<\underline{\kappa}<1<\overline{\kappa}<\infty$,
there exists at most one class $\tilde{\mathcal{K}}\in \tilde{M}$ such, that
$$
\{ \partial J_{K}=0 \}
\cap 
\{ J_{K} \leq \frac{\overline{\kappa}}{\underline{\kappa}  } \sup_{n\in \N} \overline{k}_{n}\} 
=
\emptyset
$$
for at least one representative $K\in \tilde{\mathcal{K}}$ is possible. 
\end{thm}	   
Theorem \ref{thm1} is an easy consequence of Proposition \ref{prop_instead_of_thm}.
\begin{remark}
To contextualise Theorem \ref{thm1}  consider the case $M=S^{n}$ and 
\begin{enumerate}[label=(\roman*)]
\item 
let $K$ denote a positive, monotone height function, 
to which the Kazdan-Warner obstruction,
cf. \cite{Kazdan_Warner_Scalar}, applies.
Then for any choice of spread
$$ 
\mathcal{M}
\subseteq
\mathcal{M}^{+,S,\underline{k},\overline{k}}_{i_{1},\ldots,i_{m} }(M)
\; \text{ with } \; 
K \in \mathcal{M}
$$ 
evidently $\tilde{\mathcal{M}}$ is trivial, 
as is the Morse structure $\mathcal{S}$, 
i.e. consists only of the unique class, in which $K$ is contained,  
and this class has an element, namely $K$, for which 
$\partial J_{K}=0$ is not solvable.

\item  
let 
$
\mathcal{M}\subseteq \mathcal{M}^{+,S,\underline{k},\overline{k}}_{i_{1},\ldots,i_{m} }(M)
$
be a spread consisting of pinched functions, 
for all of which by Theorem 1 in \cite{MM4}
existence is guaranteed, provided more than one critical point has negative Laplacian. Then,
\begin{enumerate}
\item[1.)] 	if $\mathcal{S}$ provides at least two local maxima, then existence of a solution to
$ \partial J_{K}=0$ is always guaranteed. In this case there is no class with lack of solvability.
\item[2.)]  if $\mathcal{S}$ provides just one maximum, then either there is no class with lack of solvability or, 
if there is, it necessarily is the class of functions, whose only critical point with negative Laplacian is the global maximum.  
\end{enumerate}
\end{enumerate}
\end{remark}

\noindent
Clearly Theorem \ref{thm1}, while non constructive, nonetheless  
\begin{enumerate}[label=(\roman*)]
\item is a mean to turn an obstruction 
to existence for one specific function - and \cite{Bourguignon} provides interesting examples - 
into existence results for entire classes of comparable functions with the caveat, 
that to run our deformation scheme always requires a description of its critical points at infinity. 
\item turns the lack of low energy solutions for one specific function 
into an a priori energy bound for solutions of entire classes, 
whose existence then follows from Theorem \ref{thm1} or by contradiction. 
\item tells us, that depending on the complexity of a Morse structure  \textit{most} positive scalar curvature candidates realising this structure are conformally prescribable in a controllable way.
\item  hints at a deeper structure on the space of conformally prescribable  scalar curvatures.
\end{enumerate} 
  
We proceed as follows.  In Section \ref{Section_The_Variational_Formulation} we quickly review the variational setting 
of the critical equation and the relation to its subcritical approximation. 
In Section \ref{Section_Ordering_And_Deformation} we clarify the structure of spreads 
and present a deformation scheme,
which applied to the relevant variational functionals of the conformally prescribed scalar curvature problem 
will eventually lead to the proof of Theorem \ref{thm1}.
In Section \ref{Sec_The_Comparison} we then show, 
how to apply the deformation scheme leading to Theorem \ref{thm1}. 
In order to clarify the theoretical interest and limitations of Theorem \ref{thm1}, 
we describe in Section \ref{Sec_Obstruction_To_Non_Existence} the classical total degree or, as it is called,  
the index counting formula argument for  existence and discuss its failure. 
Here we also study the smooth heart on $S^{3}$ more closely and prove Theorem \ref{thm2}. 
Finally in the Appendix we include for the sake of completeness some topological lemmata, 
on which the deformation scheme is based, as well as a standard remark on Morse functions, 
which shows generic non triviality of Theorem \ref{thm1}.

\

Before passing to the actual content, we point out, that the subsequent arguments, 
which we provide for $n\geq 5$, 
do have at least for $n=2,3$ less stringent adaptations and are actually easier to obtain, as we shall comment on here and there.  
For the Gaussian case we refer to  \cite{MM8} for stronger results.

\section{The Variational Formulation}\label{Section_The_Variational_Formulation}
Let us review the variational setting, while we refer to \cite{MM1},\cite{MM2},\cite{MM3} for greater detail. 
Consider a closed Riemannian manifold
\begin{equation*}
M=(M^{n},g_{0}) 
\; \text{ with } \; n\geq 5,
\end{equation*}
of positive Yamabe  invariant, i.e. 
\begin{equation}\label{Yamabe_Invariant}
Y=Y(M,[g_{0}])= 
\inf_{g\in [g_{0}]}
\frac{\int_{M}R_{g}d\mu_{g}}{\mu_{g}^{\frac{n-2}{n}}(M)}>0,
\end{equation}
where $\mu_{g}$ denotes the measure and $R_{g}$ denotes the scalar curvature on $M$ induced by $g$. 
We then may suppose $R_{g_{0}}>0$ or even $R_{g_{0}}=1$. 
In this setting the conformal Laplacian
$L_{g_{0}}=-c_{n}\Delta_{g_{0}}+R_{g_{0}}$
is a positive and self-adjoint operator and
considering a conformal metric 
$$g=g_{u}=u^{\frac{4}{n-2}}g_{0}$$  
there holds  
\begin{equation*}
\begin{split} 
d\mu_{g_{u}}=u^{\frac{2n}{n-2}}d\mu_{g_{0}}
\;  \text{ and }\; 
R=R_{g_{u}}=u^{-\frac{n+2}{n-2}}(-c_{n} \Delta_{g_{0}} u+R_{g_{0}}u) =
u^{-\frac{n+2}{n-2}}L_{g_{0}}u
\end{split}
\end{equation*}
due to conformal covariance of the conformal Laplacian, i.e.
\begin{equation*}
-c_{n}\Delta_{g_{u}}v+R_{g_{u}}v=L_{g_{u}}v=u^{-\frac{n+2}{n-2}}L_{g_{0}}(uv).
\end{equation*}
So prescribing conformally  the scalar curvature 
$R=R_{g_{u}}=K$ 
is equivalent to solving 
\begin{equation}\label{the_equation}
L_{g_{0}}u=Ku^{\frac{n+2}{n-2}}.
\end{equation} 
Evidently
$
\Vert u \Vert^2 = \Vert u \Vert_{L_{g_0}}^2 = \int u \, L_{g_{0}}u \, d\mu_{g_{0}}
$ 
defines an equivalent norm on $W^{1,2}$. We then wish to study  
on
$$
\mathcal{A}
=
\{ u\in W^{1,2}(M)\; : \; 0\not \equiv u \geq 0 \} 
$$
the scaling invariant functional
\begin{equation*}
J_{K}:\mathcal{A}\longrightarrow \R:
u\longrightarrow 
\frac
{\int  L_{g_{0}}uu d\mu_{g_{0}}}
{(\int Ku^{\frac{2n}{n-2}}d\mu_{g_{0}})^{\frac{n-2}{n}}}
\; \text{ for } \;
K>0. 
\end{equation*}
Note, that $J_{K}$ is scaling invariant, which allows us the consider
$$
X=\{ u\in \mathcal{A}\; : \; \Vert u \Vert=1 \} 
$$
as a natural variational space.  Moreover the conformal scalar curvature 
$$
R=R_{g_{u}}
\; \text{ for } \; 
g=g_{u}=u^{\frac{4}{n-2}}g_{0}
$$ 
satisfies
\begin{equation*}
r=r_{g_{u}}=\int R d\mu_{g_{u}}=\int u L_{g_{0}} ud\mu_{g_{0}},
\end{equation*}  
whence
\begin{equation*}
J_{K}(u)=\frac{r}{k^{\frac{n-2}{n}}}
\;  \text{ with } \; k = k_{g_{u}}= \int K \, u^{\frac{2n}{n-2}} d \mu_{g_{0}}. 
\end{equation*}
We then find 
\begin{equation*}
\partial J_{K}(u)v
= 
\frac{2}{k^{\frac{n-2}{n}}}
\big[\int L_{g_{0}}uvd\mu_{g_{0}}-\frac{r}{k}\int Ku^{\frac{n+2}{n-2}}vd\mu_{g_{0}}\big],
\end{equation*} 
whereby \eqref{the_equation} has variational structure. Also  $J_{K}$ is of class 
$C^{2, \alpha}_{loc}(\mathcal{A})$ 
by direct calculation. 
Then, if
\begin{equation}\label{base_assumption}
n\geq 5 
\; \text{ and } \; 
0<K \; \text{ is Morse with } \;
\{ \nabla K=0 \} \cap \{ \Delta K=0 \} = \emptyset,   
\end{equation}
there is, cf. \cite{MM3}, a one to one correspondence
\begin{equation*}
\mathbb{P}(C_{-}(K)) \xrightarrow{\;\simeq\;} C_{\infty}(J_{K})
\end{equation*}
from the \textit{subsets} of 
$$C_{-}(K)=\{\nabla K=0\}\cap \{\Delta K<0\}$$
to the pure critical points at infinity of $J_{K}$, i.e. those of type
\begin{equation*}
\{x_{1},\ldots,x_{q}\} \longrightarrow
u_{\infty,x_{1},\ldots,x_{q}} 
=  
\alpha_{1}\delta_{a_{1}}+\ldots+\alpha_{q}\delta_{a_{q}}, 
\end{equation*} 
where $\delta_{a}$ denotes a Dirac measure in $L^{\frac{2n}{n-2}}$ at $a\in M$.
Finally, if $\{ \partial J_{K} =0 \} = \emptyset $, these are the \textit{only} critical points at 
infinity, cf. Lemma \ref{lem_subcritical_transition} for the subcritical analogon.

\subsection{Subcritical Transition}\label{sec_subcritical_transition}
Let us translate these notions into the realm of subcritical approximations. As exposed in 
\cite{MM2} there is under \eqref{base_assumption} an analogous one to one correspondence
\begin{equation*}
\mathbb{P}(C_{-}(K)) \xrightarrow{\; \simeq \;} C_{\infty}(J_{K,\tau})
\end{equation*}
from $\mathbb{P}(C_{-}(K))$ to the zero weak, finite energy limiting critical points
\begin{equation*}
C_{\infty}(J_{K,\tau})
=
\{
u_{\tau} \in \{\partial J_{K,\tau}=0\}
\;:\;
u_{\tau} \xrightharpoondown{\;weakly\;} 0 \; \wedge \; 
\lim_{\tau\to 0} J_{K,\tau}(u_{\tau})<\infty
\}
\end{equation*}
of the subcritical approximation functional
$$
J_{K,\tau}(u)
=
\frac{\int L_{g_{0}}uud\mu_{g_{0}}}{(\int Ku^{p+1}d\mu_{g_{0}})^{\frac{2}{p+1}}},
\; 
p=\frac{n+2}{n-2}-\tau,\; 0\leq \tau \longrightarrow 0.
$$
These critical points are of isolated simple bubbling type 
\begin{equation}\label{type_of_subcritical_solutions}
(\tau,\{x_{1},\ldots,x_{q}\})
\longrightarrow
u_{\tau,x_{1},\ldots,x_{q}}
=
\alpha_{1}\varphi_{a_{1},\lambda_{1}}
+
\ldots 
+
\alpha_{q}\varphi_{a_{q},\lambda_{q}}
+
v,
\end{equation}
where $\varphi_{a,\lambda}$ denotes a suitable bubble, and have critical limiting energy
\begin{equation}\label{energy_comparison_critical_sub_critical}
\begin{split}
J_{K,\tau}(u_{\tau,x_{1},\ldots,x_{q}})
=
J_{K}(u_{\infty,x_{1},\ldots,x_{q}})
+
o_{\tau}(1),
\end{split}
\end{equation}
see \eqref{pure_critical_point_at_infinity} and \eqref{critical_energies} for the value. 
We will always assume \eqref{base_assumption}, cf. (i) of Definition \ref{def_M_S}, 
and in fact rely in our arguments on this close relationship 
of pure critical points at infinity and zero weak, finite energy limiting subcritical solutions. 
In particular the following non existence statement is analogous 
to the \textit{"only"} statement on critical points at infinity in case $\{ \partial J_{K} = 0 \}$ above.
\begin{lemma}\label{lem_subcritical_transition}
Suppose 
$\{\partial J_{K}=0\}\cap \{J_{K}\leq L\}=\emptyset$.
Then 
\begin{equation*}
\;\exists\; \tau_{0}>0 
\;\forall \; 0<\tau<\tau_{0} 
\; : \; 
\{\partial J_{K,\tau}=0\}\cap \{J_{K,\tau}\leq L\}=
C_{\infty}(J_{K,\tau})\cap \{J_{K,\tau}\leq L\}.
\end{equation*}
\end{lemma}
\begin{proof}
Arguing by contradiction we assume, that for some
$\tau_{m}\searrow 0$ 
there exist
\begin{equation*}
u_{\tau_{m}}
\in 
\{\partial J_{K,\tau_{m}}=0\}
\cap 
\{J_{K,\tau_{m}}\leq L\}
,\; \text{ but }\; 
u_{\tau_{m}}\not \in 
C_{\infty}(J_{K,\tau_{m}})\cap \{J_{K,\tau_{m}}\leq L\}.
\end{equation*}
Since $J_{K,\tau}$ is scaling invariant and $J_{K,\tau_{m}}(u_{\tau_{m}})\leq L$,  we may assume 
\begin{equation}\label{renormalisation}
k_{\tau_{m}}
=
\int Ku_{\tau_{m}}^{\frac{2n}{n-2}-\tau_{m}}=1
\; \text{ instead of }\; 
r_{u_{\tau_{m}}}=\int L_{g_{0}}u_{\tau_{m}}u_{\tau_{m}}=1.
\end{equation}
By virtue of Proposition 3.1 in \cite{MM1} we then find, 
that upon a subsequence
$$u_{\tau_{m}}
=
u_{\infty } 
+
\sum_{i=1}^{q}\alpha_{i_{m}}\varphi_{a_{i_{m}},\lambda_{i_{m}}}+v_{m}$$
with 
$u_{\infty}=0$ or $u_{\infty}> 0$ solving $\partial J_{K}(u_{\infty})=0$ 
and, as
$0<\tau_{m}\xrightarrow{\; m\to \infty\; } 0$, 
\begin{enumerate}[label=(\roman*)]
\item $\Vert v_{m} \Vert \longrightarrow  0$
\item $\lambda_{i_{m}}\longrightarrow \infty$
\item $M \supset (a_{i_{m}})\longrightarrow a_{i_{\infty}}$
\item 
$
c<\alpha_{i_{m}},\lambda_{i_{m}}^{\theta_{m}}<C
$
\item 
$
\frac
{
r_{\infty}K(a_{i_{\infty}})
}
{
4n(n-1)\lambda_{i_{m}}^{\theta}
}
\alpha_{i_{m}}^{\frac{4}{n-2}}
\longrightarrow 
1
$
\item $(\varepsilon_{i,j})_{m} \longrightarrow   0$
for all $i\neq j$,
\end{enumerate}
where
$0<c<C<\infty$,
 $\theta_{m}=\frac{n-2}{2}\tau_{m}$
and
\begin{equation*}
r_{\infty}=\lim_{m\to \infty}r_{u_{\tau_{m}}},\;
r_{u_{\tau_{m}}}
=
\int L_{g_{0}}u_{\tau_{m}}u_{\tau_{m}}
=
J_{K,\tau_{m}}(u_{\tau_{m}}). 
\end{equation*}
 In case $q=0$, i.e. 
 $u_{\tau_{m}}
 =
 u_{\infty } + v_{m}$, we then have 
 \begin{equation*}
 J_{K}(u_{\infty})=\lim_{m\to \infty}J_{K,\tau_{m}}(u_{\tau_{m}})\leq L
 \end{equation*}
contradicting 
$\{\partial J_{K}=0\}\cap \{J_{K}\leq L\}=\emptyset$.
In case $q>0$ and $u_{\infty}>0$ we have
\begin{equation}\label{r_u_infinity_estimate}
r_{u_{\infty}}
=
\int L_{g_{0}}u_{\infty}u_{\infty}
\leq 
\lim_{m\to \infty} \int L_{g_{0}}u_{\tau_{m}}u_{\tau_{m}}
=
r_{\infty}
\end{equation}
by positivity of $L_{g_{0}}$. Also,  as follows from testing 
$0=\partial J_{\tau_{m}}(u_{\tau_{m}})$ with $u_{\infty}$, 
\begin{equation}\label{r_infinity_k_infinity_relation}
\frac{r_{u_{\infty}}}{k_{u_{\infty}}}
=
\frac{r_{\infty}}{k_{\infty}},
\end{equation} 
while 
\begin{equation}\label{k_infinity_information}
 k_{\infty}
=\lim_{m\to \infty}\int Ku_{\tau_{m}}^{\frac{2n}{n-2}-\tau_{m}}=1
\end{equation}
due to \eqref{renormalisation}.
Then from 
\eqref{r_u_infinity_estimate},\eqref{r_infinity_k_infinity_relation} 
and 
\eqref{k_infinity_information}
we deduce
\begin{equation*}
J_{K}(u_{\infty})
=
\frac{r_{u_{\infty}}}{k_{u_{\infty}}^{\frac{n-2}{n}}}
\leq 
r_{\infty}^{\frac{2}{n}}(\frac{r_{\infty}}{k_{\infty}})^{\frac{n-2}{n}}=r_{\infty}
=
\lim_{m\to \infty}J_{K,\tau_{m}}(u_{\tau_{m}})\leq L,
\end{equation*}
contradicting
$\{\partial J_{K}=0\}\cap \{J_{K}\leq L\}=\emptyset$.
Hence $u_{\infty}=0$ and $q\geq 1$, i.e. 
\begin{equation*}
u_{\tau_{m}}\xrightharpoondown{\quad}0
\;\text{ weakly with }\; J_{K,\tau_{m}}(u_{\tau_{m}})\leq L
\end{equation*}
is of 
zero weak and finite energy limiting type. But then
\begin{equation*}
u_{\tau_{m}}\in 
C_{\infty}(J_{K,\tau_{m}})\cap \{J_{K,\tau_{m}}\leq L\}
\end{equation*}
by uniqueness of such solutions, cf. Theorem 1 in \cite{MM2}.
\end{proof}

Finally each $u_{\tau,x_{1},\ldots,x_{q}}\in C_{\infty}(J_{K,\tau})$ is non degenerate with Morse index
$$
m(J_{K,\tau}, u_{\tau, x_1, \dots, x_q}) 
=
(q-1) + \sum_{i=1}^q 
(n-m(K,x_i)) 
$$

and, as there are only finitely many, they induce a change of topology 
\begin{equation*}
J_{K,\tau}^{\;c+\varepsilon}
\simeq   
J_{K,\tau}^{\;c-\varepsilon}
\; \sharp \;
\sum_{i}\mathcal{D}_{i},\;
c=J_{K,\tau}(u_{\tau,x_{i_{1}},\ldots,x_{i_{q}}})
\end{equation*}

of the sublevel sets
$
J_{K,\tau}^{\;c\pm \varepsilon }
=
\{ J_{K,\tau} \leq c\pm \varepsilon  \} 
$
by attaching cells $\mathcal{D}_{i}$ with
$$\dim(\mathcal{D}_{i})=m(J_{K,\tau},u_{\tau,x_{i_{1}},\ldots,x_{i_{q}}}).$$ 

And likewise we have a homological equivalence 
$$
J_{K}^{\; c+\varepsilon}
\simeq 
J_{K}^{\; c-\varepsilon}
\sharp 
\sum_{i} \mathcal{D}_{i}
,\;
c=J_{K}(u_{\infty,x_{i_{1}},\ldots,x_{i_{q}}})
$$

with corresponding dimensions

\begin{equation}\label{index_J_K_Formula}
\begin{split}
\dim(\mathcal{D}_{i})
= \; & 
\ind(J_{K},u_{\infty,x_{i_{1}},\ldots,x_{i_{q}}}) \\
= \; &
m_{J_{K,\tau}}(u_{\tau,x_{i_{1}},\ldots,x_{i_{q}}})
=
(q-1) + \sum_{i=1}^q 
(n-m(K,x_i))
\end{split}
\end{equation}

for the pure critical points at infinity, cf. \cite{MM3}.

\section{Ordering and Deformation}\label{Section_Ordering_And_Deformation}

\noindent
Recall, that in case $n\geq 5$ and for a function 
$K\in \mathcal{M}^{+,S,\underline{k},\overline{k}}_{i_{1},\ldots,i_{m} }(M)$ 
the set
$$C_{-}(K)=\{ \nabla K=0 \}\cap \{ \Delta K < 0 \}$$
determines the pure critical points at infinity of $J_{K}$, cf. Section \ref{Section_The_Variational_Formulation}.
In fact for every subset
\begin{equation*}
\{x_{1},\ldots,x_{l}\}\subset C_{-}(K)
\end{equation*}
there exists exactly one critical point at infinity of type
\begin{equation}\label{pure_critical_point_at_infinity}
u_{\infty, x_{1},\ldots,x_{l}}
=
\alpha_{1}\delta_{x_{1}}+\ldots + \alpha_{l}\delta_{x_{l}}
\end{equation}
with energy, cf. \eqref{energy_strip_multi_bubbles},
\begin{equation}\label{critical_energies}
J_{K}(\alpha_{1}\delta_{x_{1}}+\ldots 
+ 
\alpha_{l}\delta_{x_{l}})
=
c_{n}(\sum_{i=1}^{l}\frac{1}{K(x_{i})^{\frac{n-2}{2}}})^{\frac{2}{n}}
,\; c_{n}>0,
\end{equation}
Likewise for $n=2,3$, but replacing subsets with elements due to single bubbling, and also the condition 
\eqref{energy_strip_multi_bubbles}
can be relaxed from subsets to elements. 
In any case equation \eqref{critical_energies} relates the variational functional $J_{K}$ 
to the notion of spreads $\mathcal{M}$ from Definition \ref{def_M_S}, which in return justifies the notion of the equivalence classes 
$\tilde{\mathcal{M}}$. 
The impact of (iv) in Definition \ref{def_M_S} is, that $\tilde{\mathcal{M}}$ becomes an ordered set.  

\subsection{The maximal critical energy}
In fact consider for $K\in \mathcal{M}, \mathcal{M}$ a spread, the maximal energy value 
\begin{equation*}
\mu(K)=\max_{C_{\infty}(J_{K})} J_{K}
\end{equation*}
over all pure critical points at infinity of $J_{K}$, i.e. those of type \eqref{pure_critical_point_at_infinity}. 
Since $n\geq 5$ and by \eqref{critical_energies}, 
the maximal value $\mu(K)$ arises as the energetic value of the unique  pure critical point at infinity 
constituted of all the critical points of $K$ with negative Laplacian, i.e. 
\begin{equation}\label{maximal_pure_value}
\mu(K)
=
\max_{C_{\infty}(J_{K})}J_{K}
=
c_{n}(\sum_{x\in C_{-}(K)}\frac{1}{K(x)^{\frac{n-2}{2}}})^{\frac{2}{n}}.
\end{equation}
We then obtain the desired ordering from  the maximal energy value
$\mu$.
\begin{lemma}\label{lem_energetic_identification}
Let $K_{1},K_{2}\in \mathcal{M},\mathcal{M}$ a spread.  Then
\begin{equation*}
K_{1}=K_{2} \; \text{ in } \;  \tilde{\mathcal{M}}
\Longleftrightarrow 
\; \exists \; i\in \N
\; : \; 
\underline{k}_{i}
\leq 
\mu(K_{1}),\mu(K_{2}) 
\leq 
\overline{k}_{i}
\end{equation*}
\end{lemma}	 
\begin{proof}
To show the lemma, we suppose
\begin{enumerate}[label=(\roman*)]
\item on one hand, that
$K_{1}=K_{2} \; \text{ in } \;  \tilde{\mathcal{M}}$, i.e.  
$\mathcal{N}(K_{1})= \mathcal{N}(K_{2})=A$ by Definition \ref{def_partition} and \eqref{the_map_N}.
Then 
$$
\mu(K_{2})=c_{n}(\sum_{k\in A}\frac{1}{K_{1}(\bar{x}_{i_{k}})^{\frac{n-2}{2}}})^{\frac{2}{n}}
\; \text{ and } \; 
\mu(K_{2})=c_{n}(\sum_{k\in A}\frac{1}{K_{2}(\bar{x}_{i_{k}})^{\frac{n-2}{2}}})^{\frac{2}{n}} 
$$
and the implication follow from \eqref{energy_strip_multi_bubbles}. 
\item on the other hand, that 
$
\underline{k}_{i}
\leq 
\mu(K_{1}),\mu(K_{2}) 
\leq 
\overline{k}_{i}$ 
for some $i\in \N$, 
which  for the subsets 
$$C_{-}(K_{1})\sim A_{1} \subseteq \{ 1,\ldots,m \} 
\; \text{ and } \; 
C_{-}(K_{2})\sim A_{2} \subseteq \{ 1,\ldots,m \} $$ 
translates by \eqref{energy_strip_multi_bubbles} into
$$
\underline{k}_{i} 
\leq 
c_{n}(\sum_{k\in A_{1}}\frac{1}{K_{1}(\bar{x}_{i_{k}})^{\frac{n-2}{2}}})^{\frac{2}{n}}
, 
c_{n}(\sum_{k \in A_{2}}\frac{1}{K_{1}(\bar{x}_{i_{k}})^{\frac{n-2}{2}}})^{\frac{2}{n}}
\leq 
\overline{k}_{i}. 
$$
Then the uniqueness condition \eqref{uniqueness_condition} implies $A_{1}=A_{2}$ and the conclusion follows.
\end{enumerate}
This proves the desired equivalence statement.  
\end{proof}

In particular the partition 
$\sim$ on $\mathcal{M}$ leading to $\tilde{\mathcal{M}}$ 
defines an injective mapping  
$$
i:\tilde{\mathcal{M}}\longrightarrow \N: \tilde K \longrightarrow i(\tilde K)
$$ 

by
$$
i(\tilde K)=l
\Longleftrightarrow 
\; \forall \; K\in \tilde K
\; : \; 
\underline{k}_{l}
\leq
\mu(K)
\leq
\overline{k}_{l}.
$$
As a consequence  we may order the equivalence classes $\tilde{\mathcal{K}}\in \tilde{\mathcal{M}}$ via 
\begin{equation*}
\tilde{\mathcal{K}_{1}}\leq \tilde{\mathcal{K}_{2}}
\Longleftrightarrow
i(\tilde{\mathcal{K}}_{1})\leq i(\tilde{\mathcal{K}}_{2}) 
,
\end{equation*}
i.e.  we induce an ordering on $\tilde{\mathcal{M}}$ according to the energy strip, 
in which we find the maximal pure critical value at infinity of representatives.
Hence and conversely we have  
from the finite set
\begin{equation*}
N 
=
\{ 1\leq j\in \N \; : \; \; i(\tilde K)=j \; \text{ for some } \; \tilde{K}\in \tilde{M} \} 
\subset
\N_{\geq 1}
\end{equation*}
a bijection $I:N\longrightarrow \tilde{\mathcal{M}}$
onto the set of equivalence classes $\tilde{\mathcal{M}}$, which we index as
$$\{ \tilde{K}_{j} \}_{j\in N}=\tilde{M}$$ 
accordingly. With this notions at hand we make
\begin{definition}\label{def_sigma}
Given $L>0$ let
\begin{equation*}
\sigma
=
\min 
\{n\in \N
\;:\;
\forall \; n<j\in N
 \; \wedge \; 
 K\in \tilde{\mathcal{K}}_{j}
\;:\;
\{\partial J_{K}=0\} \cap \{ J_{K}\leq L \} \neq \emptyset 
\}
\geq 0
.
\end{equation*}
\end{definition}

\

\noindent
Hence by definition and recalling $N\subset \N_{\geq 1}$
\begin{enumerate}[label=(\roman*)]
 \item 
if $\sigma=0$, then $\; \forall \; K \in \mathcal{M}\; : \; \{\partial J_{K}=0\} \cap \{ J_{K}\leq L \} \neq \emptyset$, 
i.e. we can always solve
 \item 
if $\sigma >0$, then $\sigma \in N$ and
$$
\forall\; \sigma <\overline{\sigma} \in N  \; \wedge \;  K_{\overline{\sigma}}\in \tilde{\mathcal{K}}_{\overline{\sigma}}
\;:\; 
\{\partial J_{K_{\overline{\sigma}}}=0\} \cap \{ J_{K_{\overline{\sigma}}}  \leq  L\} \neq \emptyset
$$
 \item  
 if $\sigma > 0$, then
$
 \exists\;  K_{\sigma}\in \tilde{\mathcal{K}}_{\sigma}\;:\;\{ \partial J_{K_{\sigma}}=0\} \cap \{ J_{K_{\sigma}}\leq L \} =\emptyset$
 
 \item 
 if
$\sigma = \min N$, then $\tilde{K}_{\sigma}$ is the unique class, for which 
$$
\; \exists \; K_{\sigma} \in \tilde{\mathcal{K}}_{\sigma}\; : \; 
\{ \partial J_{K_{\sigma}}=0 \} \cap \{ J_{K_{\sigma}}\leq L \}=\emptyset. 
$$
\end{enumerate}
In other words $\tilde{\mathcal{K}}_{\sigma}$, if existent and that means $\sigma >0$, is in terms of the order on $\tilde{\mathcal{M}}$
the \textit{highest} class of functions in $\mathcal{M}$, for which we cannot solve for \textit{at least one} representative $K_{\sigma}\in \tilde{\mathcal{K}}_{\sigma}$.

\begin{remark}
In case $n=2,3$ due to single bubbling we may likewise partition along 
$$
\mu(K)=\max_{C_{\infty}(J_{K})} J_{K}.
$$   
By (iii) or (iv) in Definition \ref{def_M_S} this means, that we
identify those functions $K\in \mathcal{M}$, whose minimal values
$
\min_{x\in C_{-}(K)} K
$ 
over their critical points with negative Laplacian lie in the same energy strip, 
and thus identify much more functions $K$ with respect to the case $n\geq 5$, where 
all critical points with negative Laplacian must lie in the same energy strips. 
However, cf. Proposition \ref{prop_comparison_argument}, 
the deformation scheme is applied to the highest energy critical point at infinity, 
which in case $n=2,3$ corresponds to the critical point of $K$ with negative Laplacian realizing $\min_{x\in C_{-}(K)} K$.
\end{remark}
\subsection{A Deformation Scheme}
We present a deformation scheme, leading by means of a comparison argument to the existence of critical points of a functional. 
The scheme is non constructive, as it follows from arguing by contradiction. 
The contradiction would be, that in absence of a critical point, whose existence we wish to prove, 
it was possible to circumvent by energetic deformation an energetically isolated, non degenerate critical point of a different functional, 
which topologically is impossible. We visualise the setting in Figure \ref{fig_deformation}. 
\begin{figure}[H]  
\centering
\includegraphics[width=0.8\textwidth]{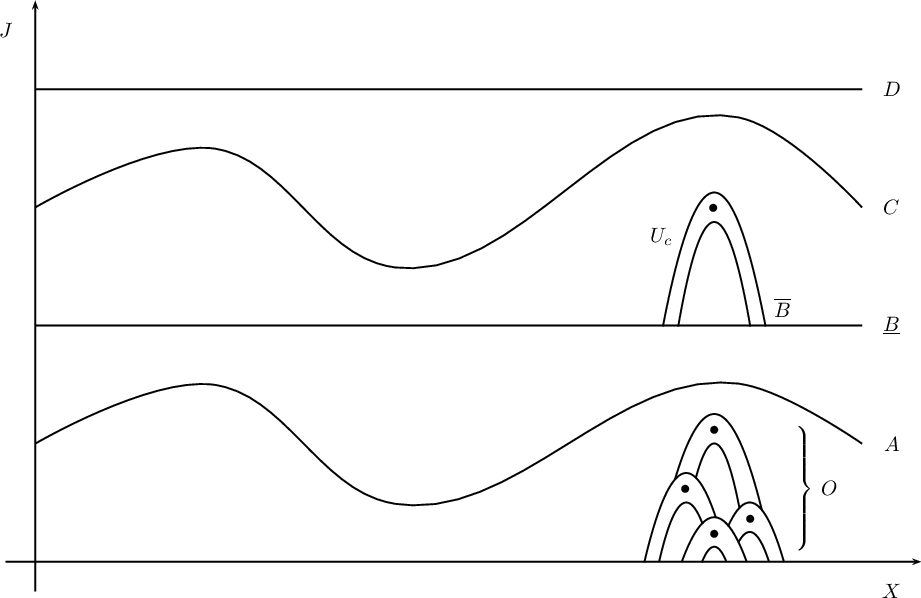} 
\caption{The deformation scheme to circumvent a critical point} 
\label{fig_deformation}
\end{figure}
\begin{proposition}[The Deformation Scheme]\label{prop_deformation_scheme}
$_{}$
\begin{enumerate}[label=(\roman*)]
\item Consider  energies $I,J$ satisfying the Palais-Smale condition and with intuitive notation sublevels 
$$
D\simeq \{ J\leq D \} , \underline B\simeq \{ J\leq \underline B \} 
\; \text{ and } \; 
C\simeq \{ I\leq C \} ,A\simeq \{ I\leq A \}
$$ 
of $J$ and $I$ respectively.
\item $J$ has exactly one non degenerate critical point 
$$c \in \{ \partial J=0 \}\cap \{ \underline B\leq J \leq D \} $$
and we let
$$\overline B=U_{c}\cup \underline B$$ 
for a neighbourhood $U_{c}$ of the unstable manifold
$W_{u}(c)$.
\item
All critical points of $J$ in $D$ are non degenerate and there holds 
$$
O
=
\cup_{x\in \{ \partial J=0 \}\setminus \{ c \}  \cap D  }
U_{x}
\subseteq A
$$
for the union $O$ of some neighbourhoods $U_{x}$ of $W_{u}(x)$.
\item There holds
$
A \subseteq \underline B \subseteq \overline B \subseteq C \subseteq D
$.
\end{enumerate}
Then necessarily 
\begin{equation}\label{critical_regime_for_I}
\{ \partial I=0 \} \cap \{ A\leq I \leq C \} \neq \emptyset.
\end{equation}
\end{proposition}	

\begin{proof}
Arguing by contradiction we may retract by weak deformation
$$
D \xhookrightarrow{\quad\;} \overline B
\; \text{ and } \; 
C \xhookrightarrow{\quad\;} A
$$
and, since $A\subseteq \overline B \subseteq C\subseteq D$, we find the homotopy equivalence relations
\begin{equation}\label{first_homotopy_equivalence}
A\simeq \overline B \simeq C \simeq D,
\end{equation} 
cf. Corollary \ref{cor_homotopy_equivalence}. On the other hand we have weak deformation retracts
$$
C \xhookrightarrow{\quad\;} A
\; \text{ and } \; 
\underline B \xhookrightarrow{\quad\;} O
$$
and, since $O\subseteq A \subseteq \underline B \subseteq C$, we find again by Corollary \ref{cor_homotopy_equivalence} the homotopy equivalence relations
\begin{equation}\label{second_homotopy_equivalence}
O\simeq A \simeq \underline B \simeq C.
\end{equation}
Then from \eqref{first_homotopy_equivalence} and \eqref{second_homotopy_equivalence} we deduce the homotopy equivalence
$\underline B\simeq \overline B$, while 
$\overline B \simeq \underline B\;\sharp\; \mathcal{D}^{\,m(J,x_{\infty})}$, 
i.e. $\overline B $ is homotopy equivalent to $\underline B$ with a cell
$\mathcal{D}$ of dimension 
$\dim(\mathcal{D})=m(J,x_{\infty})$ 
attached, which is an obvious contradiction and \eqref{critical_regime_for_I} follows.
\end{proof}

\noindent 
Applying Proposition \ref{prop_deformation_scheme} for $\underline{\sigma}<\sigma$ and suitable $\mathcal{M}$
to the  functionals
$$
I=J_{K_{\underline{\sigma}}}
\; \text{ and } \; 
J=J_{K_{\sigma}} 
\; \text{ for representatives} \;
K_{\underline{\sigma} }\in \mathcal{K}_{\underline{\sigma} }, K_{\sigma}\in \tilde{\mathcal{K}}_{\sigma}, 
$$
cf.  Definitions ,\ref{def_M_S},\ref{def_partition} and \ref{def_sigma},
then will lead to Theorem \ref{thm1}

\section{The Comparison}\label{Sec_The_Comparison}
Recall, that by construction, cf. \eqref{maximal_pure_value}, Lemma \ref{lem_energetic_identification} with the remarks following, 
the maximal critical values
$$
\mu(K)
=
c_{n}(\sum_{x\in C_{-}(K)}\frac{1}{K(x)^{\frac{n-2}{2}}})^{\frac{2}{n}}
$$
for each
$K\in \tilde{\mathcal{K}}_{i}$
are comparable by 
$$
\underline{k}_{i}
\leq
\mu(K)
\leq
\overline{k}_{i}
$$
and arise from corresponding critical points at infinity 
\begin{equation*}
u_{\infty,i} 
=
\sum_{x_{j} \in C_{-}(K)}\alpha_{j}\delta_{x_{j}}.
\end{equation*}
Moreover, as follows as Lemma \ref{lem_subcritical_transition}, if
$$
\{\partial J_{K}=0\} \cap \{ J_{K}\leq L \} =\emptyset
\; \text{ for some } \; 
L\geq \mu(K)
$$ 
then the only critical points at infinity within $\{ J\leq L \} $
are those of pure type and hence the maximal critical and the maximal pure critical value on the latter set coincide. 

\ 

Recalling Definition \ref{def_sigma} and the following comments, 
we then suppose
\begin{equation*}
\exists\;K_{\sigma}\in \tilde{\mathcal{K}}_{\sigma}\;:\;
\{\partial J_{K_{\sigma}}=0\}\cap \{ J_{K_{\sigma}}\leq L \} =\emptyset
\end{equation*}

and wish for some $\underline{\sigma}<\sigma$ to lead
\begin{equation*}
\exists\;K_{\underline \sigma}\in \tilde{K}_{\underline\sigma}\;:\;
\{\partial J_{K_{\underline \sigma}}=0\}\cap \{ J_{K_{\underline{\sigma} }}\leq L \} =\emptyset
\end{equation*}

for a suitable choice of $L\geq 0$ to a contradiction. 

\
 
\noindent
Since by construction, cf. \eqref{ordering_the_k_i},
$$
\ldots <\underline{k}_{i-1}\leq \overline{k}_{i-1}<\underline{k}_{i}<\overline{k}_{i} < \underline{k}_{i+1}\leq \overline{k}_{i+1} < \ldots,
$$
we are then in the situation, that all critical values
of $J_{K_{\underline \sigma}}$, i.e.
\begin{equation*}
J_{K_{\underline \sigma}}(C_{\infty}(J_{K_{\underline \sigma}}))
=
\{ \;
c_{n}(\sum_{x\in A}\frac{1}{K_{\underline \sigma}(x)^{\frac{n-2}{2}}})^{\frac{2}{n}}
\;:\;
A\subseteq C_{-}(K_{\underline \sigma})
\; \}
\end{equation*} 
and all critical values of $J_{K_{\sigma}}$ apart from $\mu(K_{\sigma})$, i.e.
\begin{equation*}
J_{K_{\sigma}}(C_{\infty}(J_{K_{\sigma}})\setminus \{u_{\infty,\sigma}\})
=
\{ \;
c_{n}(\sum_{x\in A}\frac{1}{K_{\sigma}(x)^{\frac{n-2}{2}}})^{\frac{2}{n}}
\;:\;
A\subsetneq C_{-}(K_{\underline \sigma})
\; \},
\end{equation*}
are below $\overline{k}_{\sigma-1}$, while $\mu(K_{\sigma})\geq \underline{k}_{\sigma}>\overline{k}_{\sigma-1}$. 
Finally let us denote by
\begin{equation*}
u_{\infty,\sigma}  
=
\sum_{x_{j} \in C_{-}(K_{\sigma})}\alpha_{j}\delta_{x_{j}}
\end{equation*}
the unique critical point at infinity of $J_{K_{\sigma}}$ with maximal energy 
$$\mu(K_{\sigma})=J(u_{\infty,\sigma})$$
and let us accordingly denote by 
\begin{equation}\label{u_tau_sigma}
u_{\tau,\sigma}
\in C_{\infty}(J_{K_{\sigma},\tau}),
\end{equation}
cf. Section \ref{sec_subcritical_transition}, 
the unique critical point of $J_{K_{\sigma},\tau}$ associated to 
$u_{\infty,\sigma}$.

\

We now provide a criterion for the existence of a critical point of $J_{K_{\underline{\sigma}}}$ 
within a certain energy strip.

\begin{proposition}\label{prop_comparison_argument}
Let $\mathcal{M}\subseteq \mathcal{M}^{+,S,\underline{k},\overline{k}}_{i_{1},\ldots,i_{m} }(M)$
be a spread. For
$ \underline{\sigma}<\sigma$
consider
$$
K_{\underline{\sigma} }\in \tilde{\mathcal{K}}_{\underline{\sigma} }
\; \text{ and } \;  
K_{\sigma}\in \tilde{\mathcal{K}}_{\sigma}
,
\; \text{ where } \; 
\tilde{\mathcal{K}}_{\underline{\sigma}}
,
\tilde{\mathcal{K}}_{\sigma}
\in
\tilde{\mathcal{M}},
$$ 
and suppose, that 
\begin{equation*}
\{ \partial J_{K_{\sigma }} = 0 \}
\cap
\{ J_{K_{\sigma}}\leq \frac{\kappa_{\sigma}}{\kappa_{\sigma-1}}\overline{k}_{\sigma} \}  
=
\emptyset
\end{equation*}
as well as 
\begin{equation}\label{energy_strip_restrictions}
\frac{K_{ \sigma}}{\kappa_{\sigma}^{\frac{n}{n-2}}}
<
K_{\underline{\sigma}}
<
\frac{K_{ \sigma}}{\kappa_{\sigma-1}^{\frac{n}{n-2}}}
\quad \text{ and } \quad 
\kappa_{\sigma} \overline{k}_{\sigma-1}<\kappa_{\sigma-1}\underline{k}_{\sigma}
\end{equation}
for some 
$0<\kappa_{\sigma-1}<1<\kappa_{\sigma}<\infty$.
Then 
$\{\partial J_{K_{\underline \sigma}}=0\}
\cap
\{
J_{K_{\underline{\sigma}}}
\leq 
\kappa_{\sigma}\overline{k}_{\sigma}
\}
\neq 
\emptyset
. 
$
\end{proposition}
\begin{proof}
Arguing by contradiction and using  Lemma \ref{lem_subcritical_transition} we may assume, that
\begin{enumerate}[label=(\roman*)]
\item[1.)] \quad 
$
\{\partial J_{K_{\sigma},\tau}=0\}
\cap 
\{J_{K_{\sigma},\tau}\leq \frac{\kappa_{\sigma}}{\kappa_{\sigma-1}}\overline{k}_{\sigma}\}
=
C_{\infty}(J_{K_{\sigma},\tau})
$
\item[2.)] \quad
$
\{\partial J_{K_{\underline\sigma},\tau}=0\}
\cap 
\{J_{K_{\underline\sigma},\tau}\leq \kappa_{\sigma}\overline{k}_{\sigma}\}
=
C_{\infty}(J_{K_{\underline\sigma},\tau})
$
\end{enumerate}
for all $0<\tau \ll 1$ sufficiently, since 
$$
\frac{\kappa_{\sigma}}{\kappa_{\sigma-1}}\overline{k}_{\sigma} > \overline{k}_{\sigma}
\; \text{ and } \; 
\kappa_{\sigma}\overline{k}_{\sigma}>\overline{k}_{\sigma}>\overline{k}_{\sigma-1}\geq \overline{k}_{\underline{\sigma} -1}.$$
Having thus passed to the subcritical scenario, by
\eqref{energy_comparison_critical_sub_critical} and \eqref{energy_strip_restrictions}
there exist 
$$
\kappa_{\tau,\sigma-1},\kappa_{\tau,\sigma}
=
\kappa_{\sigma-1},\kappa_{\sigma}
+
o_{\tau}(1)
\; \text{ and } \;
\overline{k}_{\tau,\sigma-1},\underline{k}_{\tau,\sigma}   
=
\overline{k}_{\sigma-1},\underline{k}_{\sigma}   
+
o_{\tau}(1)
$$
such, that analogously to the critical situation there holds
\begin{equation}\label{tau_relations_for_k_and_kappa}
\begin{split}
\kappa_{\tau,\sigma-1}J_{K_{\sigma},\tau}
\leq 
J_{K_{\underline \sigma},\tau}
\leq 
\kappa_{\tau,\sigma}J_{K_{\sigma},\tau}
\quad \text{ and } \quad
\kappa_{\tau,\sigma}\overline{k}_{\tau,\sigma-1}
<
\kappa_{\tau,\sigma-1}\underline{k} _{\tau,\sigma}.
\end{split}
\end{equation}
Now, for $0<\tau \ll 1$ fixed, let us translate these relations to
\begin{equation*}
\begin{split}
J=J_{K_{\sigma},\tau} 
\quad \text{ and } \quad 
I=J_{K_{\underline\sigma},\tau},
\end{split}
\end{equation*}  
in order to apply Proposition \ref{prop_deformation_scheme}.
Clearly $J$ and $I$ satisfy by virtue of \eqref{tau_relations_for_k_and_kappa}
\begin{equation}\label{I_J_comparable}
\begin{split}
\kappa_{1}J\leq I \leq \kappa_{2}J
\; \text{ for } \; 
\kappa_{1}=\kappa_{\tau,\sigma-1},
\kappa_{2}=\kappa_{\tau,\sigma}.
\end{split}
\end{equation}
Moreover,  recalling \eqref{u_tau_sigma} and 1.) above,  
let us also choose a suitable 
\begin{equation}\label{k_3}
k_{3}=\overline{k}_{\tau,\sigma}=\overline{k}_{\sigma}+o_{\tau}(1),
\end{equation}
and furthermore denote by
\begin{equation*} 
\begin{split}
\{x_{1},\ldots,x_{q}\}=\{\partial J=0\}\cap \{ J\leq \frac{\kappa_{2}}{\kappa_{1}} k_{3} \} , 
\; x_{q}=c=u_{\tau,\sigma}
\end{split}
\end{equation*}
the critical points of $J$ in $\{ J\leq \frac{\kappa_{2}}{\kappa_{1}} k_{3} \}$
corresponding to $C_{\infty}(J_{K_{\sigma},\tau}),$
for which 
\begin{equation}\label{critical_energies_J}
\begin{split}
J(x_{1}) \leq  \ldots \leq J(x_{q-1})
<
k_{1}
<
k_{2}
<
J(x_{q})
\; \text{ for } \; 
k_{1}=\overline{k}_{\tau,\sigma-1},k_{2}=\underline{k}_{\tau,\sigma}.
\end{split}
\end{equation} 
Also by \eqref{tau_relations_for_k_and_kappa} there holds
\begin{equation}\label{kappa_k_relation_in_the_proof}
\kappa_{2}k_{1}
<
\kappa_{1}k_{2}. 
\end{equation}
Let us now show for 
\begin{enumerate}
\item[($\alpha$)] $\{J\leq \frac{\kappa_{2}}{\kappa_{1}}k_{3}\}=D=\{ J\leq D \} $ 
\item[($\beta$)] $\{ I\leq \kappa_{2}k_{3} \}=C=\{ I\leq C \}  $
\item[($\gamma$)] 	$\{ J\leq k_{2} \}=\underline{B}=\{ J \leq \underline{B}  \}   $
\item[($\delta$)]  $\overline{B}=U_{c}\cup \underline{B},\; c=x_{q}=u_{\tau,\sigma}$
\item[($\varepsilon $)] $\{ I\leq \kappa_{1}k_{2}\ \}=A=\{ I\leq A \}  $
\end{enumerate}
applicability of Proposition \ref{prop_deformation_scheme}. 

\

\noindent 
We first observe, that $J$ and $I$ as subcritical functionals satisfy the Palais-Smale condition, while the critical points of $J$ in $D$ are all non degenerate, cf. Section \ref{sec_subcritical_transition}, and 
$$c=x_{q}=u_{\tau,\sigma} \in \{ \underline{B}\leq J \leq D  \}$$
is the unique critical point of $J$ therein, since without loss of generality
$$
\underline{B}=k_{2}=\underline{k}_{\tau,\sigma} 
<
J(c)=J_{K_{\sigma}}(u_{\tau,\sigma})
< 
\overline{k}_{\tau,\sigma} = k_{3}\leq \frac{\kappa_{2}}{\kappa_{1}}k_{3}=D.
$$
We are thus left with verifying
$
O\subseteq A \subseteq \underline B \subseteq \overline B \subseteq C \subseteq D.
$
Indeed
\begin{enumerate}[label=(\roman*)]
\item $D\supseteq C$ follows by \eqref{I_J_comparable}.
\item $C\supseteq \overline{B}$ follows from $J(c)<k_{3}$ and
$I\leq \kappa_{2}k_{3}$
for
$J\leq k_{3}$
by \eqref{I_J_comparable}.
\item $\overline{B}\supseteq \underline{B}$ is trivial.
\item $\underline{B}\supseteq A $
follows from  \eqref{I_J_comparable}
\item
$A\supseteq O$ follows, since by \eqref{critical_energies_J} we may assume $\sup_{O}J \leq k_{1}$, and due to 
\eqref{I_J_comparable} and  \eqref{kappa_k_relation_in_the_proof}.

$$
J\leq k_{1} \Longrightarrow I \leq \kappa_{2}k_{1}
<
\kappa_{1}k_{2}
$$
\end{enumerate}
Therefore Proposition \ref{prop_deformation_scheme} is applicable and we conclude
\begin{equation*}
\{ \partial I=0 \} \cap \{ A\leq I \leq C \} \neq \emptyset,
\end{equation*}
i.e. that $\partial J_{K_{\underline{\sigma}},\tau}=0$ permits a solution 
$u\in \{ \partial J_{K_{\underline{\sigma} },\tau}=0 \} $ with energy

\begin{equation}\label{Energy_Bound_On_The_Constructed_Solution}
\kappa_{\tau,\sigma-1}\underline{k}_{\tau,\sigma} 
=
\kappa_{1}k_{2}
\leq 
J_{K_{\underline{\sigma}},\tau}(u)
\leq
\kappa_{2}k_{3}
=
\kappa_{\tau,\sigma}\overline{k}_{\tau,\sigma},
\end{equation}
cf. ($\beta$),($\varepsilon$),\eqref{I_J_comparable},\eqref{k_3}
and \eqref{critical_energies_J}  above.

\

\noindent
On the other hand, cf. 2.) above, the only critical points of  $J_{K_{\underline{\sigma}},\tau}$ in 
$
\{J_{K_{\underline{\sigma}},\tau}
\leq
\kappa_{\sigma}\overline{k}_{\sigma}
\}
$
arise from $C_{\infty}(J_{K_{\underline \sigma},\tau})$ and by construction 
\begin{equation}\label{energy_bound_from_lower_solutions}
\sup_{C_{\infty}(J_{K_{\underline \sigma},\tau})} J_{K_{\underline{\sigma}},\tau}
\leq 
\overline{k}_{\tau,\underline{\sigma} }
\leq 
\overline{k}_{\tau,\sigma-1}.
\end{equation}  
But from \eqref{energy_strip_restrictions} or \eqref{tau_relations_for_k_and_kappa}
$$
\overline{k}_{\tau,\sigma-1}
<
\kappa_{\tau,\sigma}^{-1}\kappa_{\tau,\sigma-1}\underline{k} _{\tau,\sigma}
<
\kappa_{\tau,\sigma-1}\underline{k} _{\tau,\sigma},
$$
since 
$\kappa_{\tau,\sigma}=\kappa_{\sigma}+o_{\tau}(1)$
and
$\kappa_{\sigma}>1$ 
by assumption. Hence from \eqref{Energy_Bound_On_The_Constructed_Solution} and \eqref{energy_bound_from_lower_solutions} we find
$$
u\in \{ \partial J_{K_{\underline{\sigma} },\tau}=0 \} \setminus C_{\infty}(J_{K_{\underline \sigma},\tau}),
$$
and the claim follows from passing to the limit $\tau \longrightarrow  0$.
\end{proof}

\begin{proposition}\label{prop_instead_of_thm}
Let  
$\mathcal{M}
\subseteq
\mathcal{M}^{+,S,\underline{k},\overline{k}}_{i_{1},\ldots,i_{m} }(M)
$
be a spread. 
Then, if
\begin{equation*}
\frac{\overline{\kappa}}{\underline{\kappa}}
<
\inf_{i\in \N}\frac{\underline{k}_{i} }{\overline{k}_{i-1}}
\; \text{ and } \; 
\; \forall \; K_{1},K_{2}\in \mathcal{M}
\; : \; 
\frac{K_{1}}{\overline{\kappa}^{\frac{n}{n-2}}}
<
K_{2} 
<
\frac{K_{1}}{\underline{\kappa} ^{\frac{n}{n-2}}}
\end{equation*}
for some 
$0<\underline{\kappa}<1<\overline{\kappa}<\infty$,
there exists at most one class $\tilde{\mathcal{K}}_{\sigma}\in \tilde{M}$ such, that
$$
\{ \partial J_{K_{\sigma}}=0 \}
\cap 
\{ J_{K_{\sigma}} \leq \frac{\overline{\kappa}}{\underline{\kappa}  } \max_{n\in N} \overline{k}_{n}\} 
=
\emptyset
$$
for at least one representative $K\in \tilde{\mathcal{K}}_{\sigma}$ is possible. 
\end{proposition}	

\begin{proof}
Recalling Definition \ref{def_sigma} and the remarks thereafter, there is nothing to show 
in cases $\sigma=0$ or $\sigma=\min N$. Hence we may assume, 
that there 
$\; \exists \; \underline{\sigma}<\sigma$ with $\underline{\sigma},\sigma \in N$.  
We then have to show
$$
\; \forall \; K_{\underline{\sigma} }\in \mathcal{K}_{\underline{\sigma}}
\; : \; 
\{ \partial J_{K_{\underline{\sigma} }} =0 \}  
\cap
\{ J_{K_{\underline{ \sigma}}} \leq \frac{\overline{\kappa}}{\underline{\kappa}  } \max_{n\in N} \overline{k}_{i}\} 
\neq 
\emptyset. 
$$
But this follows from Proposition \ref{prop_comparison_argument}, provided the conditions
\begin{enumerate}[label=(\roman*)]
\item 	
$\sfrac{K_{ \sigma}}{\kappa_{\sigma}^{\frac{n}{n-2}}}
<
K_{\underline{\sigma}}
<
\sfrac{K_{ \sigma}}{\kappa_{\sigma-1}^{\frac{n}{n-2}}}
$
\item
$
\kappa_{\sigma} \overline{k}_{\sigma-1}<\kappa_{\sigma-1}\underline{k}_{\sigma}
$
\item
$
\frac{\kappa_{\sigma}}{\kappa_{\sigma-1}}\overline{k}_{\sigma}
\leq
\frac{\overline{\kappa}}{\underline{\kappa}}\max_{i\in N}\overline{k}_{i}
$
\end{enumerate}
are satisfied for some 
$0<\kappa_{\sigma-1}<1<\kappa_{\sigma}<\infty$.
And we simply choose 
$\kappa_{\sigma-1}=\underline{\kappa}
\; \text{ and } \; 
\kappa_{\sigma}=\overline{\kappa}$. 
\end{proof}

\section{Obstructions to Non Existence}
\label{Sec_Obstruction_To_Non_Existence}
Let us discuss in case of a positive Yamabe invariant some standard arguments to prove  existence of solutions to $\partial J_{K}=0$, i.e. the conformally prescribed scalar curvature problem. In fact their failure does shed some light on, why Theorem \ref{thm1} is of theoretical interest. 
\subsection{The Euler Characteristic}  
Assuming non existence, i.e.
$\{\partial J_{K}=0\}=\emptyset,$ we have contractibility of 
\begin{equation*}
J_{K}^{L}=\{J_{K}\leq L\} 
\; \text{ for some } \; L\gg \mu(K) \; \text{ sufficiently large}, 
\end{equation*}
cf. \eqref{maximal_pure_value}.
Indeed we may first for some $\varepsilon >0$ retract by strong deformation 
$J_{K}^{L} \xhookrightarrow{\;sdr\;} J_{K}^{\mu(K)+ \varepsilon }$
and then perform a star shaped homotopy 
\begin{equation*}
J_{K}^{\mu(K)+\varepsilon }\simeq \{u_{0}\} \; \text{ for some fixed }\; u_{0}\in J_{K}^{\mu(K)},
\end{equation*}
for instance 
\begin{equation*}
H
:
J_{K}^{\mu(K)} \times [0,1]
\longrightarrow 
J_{K}^{L}
:
(u,t)\longrightarrow \frac{(1-t)u+tu_{0}}{\Vert (1-t)u+tu_{0} \Vert},
\end{equation*}
which is well defined, provided $L\gg \mu(K)$ is sufficiently large. 
We conclude
\begin{lemma}\label{lem_contractibility} 
 Let $K>0$ be a Morse function satisfying 
 \begin{equation*}
\{\nabla K=0\}\cap\{\Delta K=0\}=\emptyset.
 \end{equation*}
Then $J_{K}^{L}$ is contractible for all $L\gg \mu_{K}$ sufficiently large, provided 
$$\{\partial J_{K}=0\}=\emptyset.$$
In particular  $\chi(J_{K}^{L})=1$ for the Euler characteristic. 
\end{lemma}
On the other hand we may compute the latter Euler characteristic by Morse theory and obtain under the assumptions of Lemma \ref{lem_contractibility}
\begin{equation}\label{index_counting_formula}
1=\chi(J_{K}^{L})=\sum_{c\in C_{\infty}(J_{K})}(-1)^{\ind(J_{K},c)}.
\end{equation}
Generally this identity may be violated, for instance in low dimensions or for higher flatness of $K$ near its critical points and in this case, 
arguing by contradiction, necessarily at least one solution to $\partial J_{K}=0$ has to exist, 
cf.  \cite{Bahri_Coron_S_3},\cite{Ben_Ayed_M_4},\cite{Chang_S_2},\cite{Li_Part_1},\cite{Li_Part_2}. The
underlying reason is, that in these cases 
$C_{\infty}(J_{K})$ 
corresponds to a proper subset 
$Q \varsubsetneq \mathbb{P}(C_{-}(K)),$
i.e. not every combination 
$$\{x_{1},\ldots,x_{q}\}\subset C_{-}(K)=\{\nabla K=0\}\cap \{\Delta K=0\}$$
gives rise to a critical point at infinity. For instance, recalling \eqref{Index_Counting_Formula_S_2} and \eqref{index_countring_formula_n_3},
in case $M=S^{2}$ or $M=S^{3}$, due to single bubbling and fine blow-up analysis, 
$$
C_{-}(K) \xrightarrow{\;\;\simeq \;\; } C_{\infty}(J_{K}) \; : \; x \longrightarrow \delta_{x} ,
$$
i.e. by bubbling each \textit{element} of $C_{-}(K)$ induces exactly one pure critical points at infinity, 
while in assumed absence of solutions to $\partial J_{K}=0$ there are non no others, and
\begin{equation*}
\ind(J_{K},\delta_{x})=-m(K,x).
\end{equation*}
Thereby the index counting formulae \eqref{Index_Counting_Formula_S_2}  and \eqref{index_countring_formula_n_3}
easily follow from \eqref{index_counting_formula}. 

\
 
In our setting, i.e. under \eqref{base_assumption}, 
however the index counting formula \eqref{index_counting_formula} always
\footnote{This was pointed out to us by Prof. M. Ahmedou, cf. \cite{Ahmedou_Degree_Zero}}
holds true and no existence result can be obtained from that identity.
\begin{lemma}\label{lem_index_counting_trivial}
 Let $n\geq 5$ and $K>0$ be a Morse function satisfying 
 \begin{equation*}
\{\nabla K=0\}\cap\{\Delta K=0\}=\emptyset.
 \end{equation*}
Then 
$1=\sum_{c\in C_{\infty}(J_{K})}(-1)^{\ind(J_{K},c)}$.
\end{lemma}
\begin{proof}
From the one to one
\footnote{See \cite{Ahmdou_Non_Simple_Blow_Up} for a case of non simple blow-up}
 correspondence 
$\mathbb{P}(C_{-}(K)) \xrightarrow{\; \simeq \;} C_{\infty}(J_{K})$
we have
\begin{equation*}
\sum_{c\in C_{\infty}(J_{K})}(-1)^{\ind(J_{K},c)}
=
\sum_{\{x_{1},\ldots,x_{q}\}\subseteq C_{-}(K)}
(-1)^{\ind(J_{K},u_{\infty,x_{1},\ldots,x_{q}})}.
\end{equation*}
Choosing $y\in \{K=\max K\}\subseteq C_{-}(K)$, we decompose 
$
C_{-}(K)
=
C_{1}+C_{2}+C_{3}
$,
where
\begin{enumerate}[label=(\roman*)]
 \item $C_{1}=\{y\}$
 \item $C_{2}=\{A_{-}\subset C_{-}(K)\;:\; y \not \in A_{-}\}$
  \item $C_{3}=\{A_{+}\subseteq C_{-}(K)\;:\;  y\in A_{+}\neq \{y\}\}$.
\end{enumerate}
We then have a bijection
$
C_{2}\xrightarrow{\;\simeq \; } C_{3}
: \{x_{1},\ldots,x_{q}\}\longrightarrow \{x_{1},\ldots,x_{q},y\}
$ 
and there holds, cf \eqref{index_J_K_Formula},
\begin{enumerate}[label=(\roman*)]
 \item 
 $
 \ind(J_{K},u_{\infty,x_{1},\ldots,x_{q}})
 =
 (q-1)+\sum_{i=1}^{q}(n-m(K,x_{i}))
 $
 \item 
 $
 \ind(J_{K},u_{\infty,x_{1},\ldots,x_{q},y})
 =
q+(\sum_{i=1}^{q}(n-m(K,x_{i}))+(n-m(K,y))).
$
\end{enumerate}
In particular, since $y$ is a maximum of $K$, we have 
$m(K,y)=n$ and obtain 
\begin{equation*}
\ind(J_{K},u_{\infty,x_{1},\ldots,x_{q},y})
=
\ind(J_{K},u_{\infty,x_{1},\ldots,x_{q}})+1.
\end{equation*}
We thus find cancellation of $C_{2}$ with $C_{3}$ in the sense, that
\begin{equation*}
\begin{split}
\sum_{c\in C_{\infty}(J_{K})}(-1)^{\ind(J_{K},c)}
= &
\sum_{\{x_{1},\ldots,x_{q}\}\subseteq C_{1}+C_{2}+C_{3}}
\hspace{-12pt}
(-1)^{\ind(J_{K},u_{\infty,x_{1},\ldots,x_{q}})}
= 
(-1)^{\ind(J_{K},u_{\infty,y})},
\end{split}
\end{equation*}
whence the claim follows due to 
$\ind(J_{K},u_{\infty,y})=n-m(K,y)=0$.   
\end{proof}
Hence \textit{this} standard argument for existence based on the computation of the Euler characteristic by showing contractibility on the one hand and counting Morse indices on the other
is not always feasible, which is to say, that the computation of the Euler characteristic does not provide an obstruction to non existence.
However, even if \eqref{index_counting_formula} holds true or in other words, if the total degree is zero, one may still obtain existence of solutions by contradiction arguments, if  
\begin{enumerate}[label=(\roman*)]
\item some finer homological identities would is violated, cf. \cite{Ahmedou_Degree_Zero}.
\item a min-max scheme leads to a critical point at infinity of certain indices, which is not there, cf.
\cite{Bianchi_Min_Max},\cite{MM4},\cite{MM8} and Theorem \ref{thm2}.
\item the comparability of sublevels of different functionals are incompatible with non existence assumptions, as we argue here and in \cite{MM4}.
\end{enumerate}

\subsection{Connecting Orbits}\label{section_connecting_orbits}

Let us consider the smooth heart $S\simeq S^{2}$ as depicted in Figure \ref{fig_smooth_heart} on $S^{3}$, 
which is perhaps the most simple, yet ambiguous example.  
Every argument clarifying this situation will - in our opinion - be a cornerstone in the future development, not only within, but also beyond the realm of scalar curvature problems. 

\

Let us describe upon a choice of orientation the Morse homology over 
$\mathbb{Z}_{2}$.
\begin{enumerate}
\item[0)]
$
H_{0}(S^{3})
= 
\sfrac{\text{ker}\partial_{0}}{\text{Im}\partial_{1}}
=
\langle x_{0} \rangle_{\mathbb{Z}_{2}}
=
\mathbb{Z}_{2},
$
since $\partial_{1}x_{1}=0$.
\item[1)]
$
H_{1}(S^{3})
=
\sfrac{\text{ker}\partial_{1}}{\text{Im}\partial_{2}}
=
\sfrac
{\langle x_{1}\rangle_{\mathbb{Z}_{2}}}
{\langle x_{1}\rangle_{\mathbb{Z}_{2}}}
=
\mathbb{0}$, 
since $\partial_{2}x_{2}^{1}=x_{1}$ or $\partial_{2}x_{2}^{2}=-x_{1}$ 
\item[2)]
$
H_{2}(S^{3})
=
\sfrac{\text{ker}\partial_{2}}{\text{Im}\partial_{3}}
=
\sfrac
{\langle x_{2}^{1}+x_{2}^{2}\rangle_{\mathbb{Z}_{2}}}
{\text{Im}\partial_{3}}
=
\mathbb{0}$, 
since $\partial_{2}(x_{2}^{1}+x_{2}^{2})=\mathbb{0}$ and 
\begin{equation}\label{necessary_connection}
x_{2}^{1}+x_{2}^{2}\in \text{Im}\partial_{3}.
\end{equation}
To see \eqref{necessary_connection} note, that $x_{2}^{1},x_{2}^{2}$ have each exactly one positive direction orthogonal to $S$ pointing in- and outside of $S$. Following these directions, we must from both $x_{2}^{1}$ and $x_{2}^{2}$ arrive with a unique, single orbit at each maxima $x_{3}^{1},x_{3}^{2}$, cf. \eqref{height_relations}. As a consequence 
\begin{equation}\label{single_connection_3_2}
\partial_{3}x_{3}^{1}=(x_{2}^{1}+ x_{2}^{1})
\; \text{ and } \; 
\partial_{3}x_{3}^{2}=-(x_{2}^{1} +  x_{2}^{1})
\end{equation}

\item[3)]
$
H_{3}(S) 
=
\sfrac{\text{ker}\partial_{3}}{\text{Im}\partial_{4}}
=
\text{ker}\partial_{3}
=
\mathbb{Z}_{2},
$ since clearly $\text{Im}\partial_{4}=\mathbb{0}$ and by \eqref{single_connection_3_2}
$$\partial_{3}(x_{3}^{1}+x_{3}^{2})=\mathbb{0}.$$
\end{enumerate}

Passing to the prescribed scalar curvature functional $J_{K}$, suppose, that we cannot solve $\partial J_{K}=0$, while $K>0$ Morse satisfies 
\begin{equation}\label{Non_Degeneracy_Smooth_Heart}
\{ \nabla K=0 \} \cap \{ \Delta K=0 \} = \emptyset.
\end{equation}
Then by \cite{Bahri_Coron_S_3} 
the critical points at infinity of $J_{K}$ correspond exactly to single Diracs at elements of 
$$
C_{-}(K)=\{ \nabla K=0 \} \cap \{ \Delta K<0 \}  
$$
with index 
$\ind(J_{K})=m(K^{-1})$. 
Since 
$\Delta K(x_{0})<0<\Delta K(x_{3}^{1}),\Delta K(x_{3}^{2})$ 
due to \eqref{Non_Degeneracy_Smooth_Heart},  
\eqref{index_counting_formula} leaves us with 

\begin{enumerate}
\item[1.)] $C_{-}(K)=\{ x_{0} \}$
\item[2.)] $C_{-}(K)=\{ x_{0},x_{1},x_{2}^{1} \} $
\item[3.)] $C_{-}(K)=\{ x_{0},x_{1},x_{2}^{2} \} $
\end{enumerate}
as possibilities for non existence 
and 3.) is ruled out by Theorem \ref{thm2}.  
\begin{proof}[\textbf{Proof of Theorem \ref{thm2}}] 
Passing to the subcritical problem, first note, that for
$$ \label{proof_thm_2}
u_{\tau,x_{2}^{2}}\in \{ \partial J_{K,\tau}=0 \} 
\; \text{ and } \; \tau \gg 1,
$$
cf. \eqref{type_of_subcritical_solutions},
there holds
\begin{equation}\label{inject_into_smaller_sublevel}
J_{K,\tau}(\varphi_{x_{2}^{1},\lambda})
<
J_{K,\tau}(u_{\tau,x_{2}^{2}})
\end{equation}
for a suitable choice of $\lambda \simeq \tau^{-\frac{1}{2}}\gg 1$. 
Clearly \eqref{inject_into_smaller_sublevel} holds in case of strict inequality in (i),
but also follows easily in case of equality by expansion from Proposition 5.1 in \cite{MM1} using, that there
\begin{enumerate}[label=(\roman*)]
\item 	$\varepsilon_{i,j}=0$, as $q=1$
\item   $H_{i}=0$, as $M=S^{3}$
\item   $\vert \partial J_{\tau}\vert=\vert\partial J_{K,\tau}\vert=0$  at a solution or, if $v=0$.
\end{enumerate}
With this in mind, we first prove the existence of a solution
\begin{equation}\label{existence_of_solution_smooth_heart}
u_{\tau,x_{0}},u_{\tau,x_{1}},u_{\tau,x_{2}^{2}} 
\neq 
y
\in 
\{ \partial J_{K,\tau}=0\} 
\cap 
\{ J_{K,\tau}\leq J_{K,\tau}(u_{\tau,x_{2}^{2}}) \}.
\end{equation}
In fact and arguing by contradiction suppose
$$\{ u_{\tau,x_{0}},u_{\tau,x_{1}},u_{\tau,x_{2}^{2}} \} 
= 
\{ \partial J_{K,\tau}=0 \} \cap \{ J_{K,\tau}\leq J_{K,\tau}(u_{\tau,x_{2}^{2}})\}.
$$
By virtue of \eqref{inject_into_smaller_sublevel} we then may inject by bubbling for some 
$
\sigma < J_{K,\tau}(u_{\tau,x_{2}^{2}})
$
\begin{equation}\label{injection}
i:B^{2}_{1}(0)\simeq \overline{W_{u}^{K^{-1}}(x_{2}^{1})}
\longrightarrow 
 \{ J_{K,\tau}\leq \sigma \} 
 :
 x\longrightarrow \alpha_{x}\varphi_{a_{x},\lambda_{x}}+v_{x}
\end{equation}
for suitable, continuous maps $x\longrightarrow \alpha_{x},a_{x},\lambda_{x},v_{x}$ and under this map
\begin{equation}\label{immersion_of_unstable_manifolds}
i(\overline{W_{u}^{K^{-1}}(x_{1})})
\simeq
\overline{W_{u}(u_{\tau,x_{1}})}
\end{equation}
are homeomorphic. In fact  \eqref{immersion_of_unstable_manifolds} is immediate from the Morse structure of
$K$ around $x_{1}$ and of $J_{K,\tau}$ around $u_{\tau,x_{1}}$, cf. \cite{MM2}. 
On the other hand
$$
\{ J_{K,\tau}\leq \sigma \}
\xhookrightarrow{\; wdr\;} \overline{W_{u}(u_{\tau,x_{1}})}
$$
as a weak deformation retract, if we deform $\{ J\leq \sigma \}$ along 
the negative gradient flow locally patched with a strong deformation retracts  
\begin{enumerate}[label=(\roman*)] 
\item[a)]  onto $W_{u}(u_{\tau,x_{1}})$ close to $u_{\tau,x_{1}}$ 
\item[b)]  onto $u_{\tau,x_{0}}$, e.g. induced by flowing close to, but not at $u_{\tau,x_{0}}$ along 
$
\partial_{t}u
=
-
\frac{\nabla J_{K,\tau}}{\Vert \nabla J_{K,\tau} \Vert }(u),
$  
\end{enumerate}
both of which leave
$\overline{W_{u}(u_{\tau,x_{1}})}$ invariant as well.
From \eqref{injection} and \eqref{immersion_of_unstable_manifolds}
we then obtain a continuous map
$$
r:B^{2}_{1}(0)\longrightarrow S^{1}
\; \text{ with } \;
id_{S^{1}}\simeq r\lfloor_{S^{1}}: S^{1}\longrightarrow S^{1}
$$
homotopic to the identity on $S^{1}\simeq \overline{W_{u}(u_{\tau,x_{1}})}$ in contradiction to Brouwer's fixed point theorem.
Hence \eqref{existence_of_solution_smooth_heart} and,
 letting $\tau \longrightarrow 0$, the assertion follows. 
\end{proof}

\begin{remark} We remark, that
\begin{enumerate}[label=(\roman*)]
\item 
$K(x_{2}^{1})=K(x_{2}^{2})$ and therefore coinciding potentially critical values for $J_{K}$ are permitted in Theorem \ref{thm2}, but not in Theorem \ref{thm1}.

\item there is another way of seeing Theorem \ref{thm2}.
Arguing by contradiction,  the functional $J_{K,\tau}$ is Morse with exactly two Morse homological cycles
$\mathfrak{c}_{0},\mathfrak{c}_{1}\in \{ J_{K,\tau} \leq \sigma\}$
on the latter set, which are induced by $u_{\tau,x_{0}}$ and $u_{\tau,x_{1}}$. In particular
$$\mathfrak{M}_{1}(\{ J_{K,\tau}\leq \sigma \} )=\mathbb{Z}_{2}$$ 
for the first Morse homology group.  
But Morse and singular homology coincide, since the critical points of $J_{K,\tau}$ have finite indices, cf. \cite{Abbondandolo}. 
And in singular homology the only possible cycle
$$c_{1}\in \textnormal{ker}\, \partial_{1}$$  
is up to homotopy the unstable manifold of
$u_{\tau,x_{1}}$, which is however the boundary of the 2-cell 
$$c_{2}=i(\overline{W^{K^{-1}}_{u}(x_{2}^{1})}).$$ 
Hence 
$H_{1}(\{ J_{K,\tau} \leq \sigma \} )=\mathbb{0}$ and this is a contradiction. 
\end{enumerate}
\end{remark}

Theorem \ref{thm2} leaves us with the possibilities 1.) and 2.) and both are variationally stable, as follows.
\begin{enumerate}
\item[1.)] In this case $u_{\infty,x_{0}}$ is the only critical point at infinity of $J_{K}$, which acts as a minimum. 
Correspondingly $u_{\tau,x_{0}}$ is a minimum for the subcritical approximation $J_{K,\tau}$
and the only energy bounded, zero weak limit critical point.
\item[2.)] Also in this case we see variational stability \textit{at infinity} for the functional $J_{K}$
as well as for its subcritical approximation $J_{K,\tau}$. 
\begin{enumerate}[label=(\roman*)]
\item 
Consider some $\sigma< J_{K}(\alpha_{x_{2}^{2}}\delta_{x_{2}^{2}})$ and for $0<\varepsilon \ll 1 $
$$
u_{0}
=
\alpha_{0}\varphi_{a_{0},\lambda_{0}}+v_{0}
\in
V(1,\varepsilon)\cap \{ J\leq \sigma \} 
\; \text{ with } \; 
\Vert v_{0} \Vert \ll \frac{1}{\lambda_{0}^{2}}.
$$
Note, that the latter set is a neighbourhood of
$u_{\infty,x_{0}},u_{\infty,x_{1}}$ and $u_{\infty,x_{2}^{1}}$. Given thus $u_{0}$ as an initial data
consider a flow line
$$
\partial_{t}u=-\nabla J_{K}(u)
\; \text{ with } \; 
u(0)=u_{0}.
$$
Then, as long as $u$ remains in some neighbourhood
$ 
V(1,\epsilon)
$ with $
0\ll \varepsilon \ll \epsilon \ll 1
$
we have with positive constants $c_{1},c_{2},c_{3}>0$
\begin{enumerate}[label=(\roman*)]
\item[($\alpha$)]
$
\dot\alpha
=
O(\frac{\vert \nabla K(a) \vert^{2} }{\lambda^{2}}+\frac{1}{\lambda^{4}}
+
\Vert v \Vert^{2}
)
$ 	
\item[($a$)]
$
\lambda \dot a
=
-
c_{1}\frac{\nabla K(a)}{K^{\frac{5}{4}}(a)\lambda}+o(\frac{1}{\lambda})
+
O(\Vert v \Vert^{2})
$
\item[($\lambda$)]
$
\frac{\dot \lambda}{\lambda}
=
-c_{2}\frac{\Delta K(a)}{K^{\frac{5}{4}}(a)\lambda^{2}}
+
o(\frac{1}{\lambda^{2}})
+
O
(
\frac{\vert \nabla K(a) \vert^{2} }{\lambda^{2}}
+
\Vert v \Vert^{2}
)
$
\item[($v$)]
$
\partial_{t}\Vert v \Vert^{2}
\leq 
-c_{3}\Vert v \Vert^{2}
+
O(\frac{\vert \nabla K(a) \vert^{2} }{\lambda^{2}}+\frac{1}{\lambda^{4}})
$.
\end{enumerate}
In fact ($\alpha$) follows from (100) in \cite{Bahri_Coron_S_3},
since 
$c<\alpha<C$
due to $\Vert u \Vert=1$, and by scaling invariance
$$
\partial J(u)\alpha \varphi_{a,\lambda}
=
\partial J(u)u
-
\partial J(u)v
=
O(\vert \partial J(u) \vert^{2}+\Vert v \Vert^{2} ),
$$
while we have the estimates
\begin{equation}\label{rough_expansion_d_J_u}
\vert \partial J(u) \vert
=
\vert \partial J(\alpha \varphi_{a,\lambda}) \vert  
+
O(\Vert v \Vert)
\end{equation}
and, as follows by direct calculation, 
\begin{equation}\label{derivate_estimate_at_pure_bubble}
\vert \partial J(\alpha \varphi_{a,\lambda}) \vert 
=
O
(
\frac{\vert \nabla K(a) \vert }{\lambda}
+
\frac{1}{\lambda^{2}}
).
\end{equation}
Secondly ($a$),($\lambda$) follow from Lemma 6 in 
\cite{Bahri_Coron_S_3}  
using again 
\eqref{rough_expansion_d_J_u} 
and 
\eqref{derivate_estimate_at_pure_bubble} 
above to estimate $\vert \partial J(u) \vert $. 
Finally to verify $(v)$ we calculate 
\begin{equation*}
\begin{split}
\partial_{t} \Vert v \Vert^{2}
= &
\langle \partial_{t}v,v\rangle
=
\langle \partial_{t}u,v\rangle
=
-\partial J(u)v 
\leq 
-\partial J(\alpha \varphi_{a,\lambda})v-\tilde{c}_{2}\Vert v \Vert^{2} \\
\leq &
-c_{2}\Vert v \Vert^{2}
+
O
(
\frac{\vert \nabla K(a) \vert }{\lambda}
+
\frac{1}{\lambda^{2}}
)
\end{split}
\end{equation*}

using \eqref{derivate_estimate_at_pure_bubble} above and from \cite{Bahri_Coron_S_3}
Lemma A2 and (Vo).

\

Hence ($\alpha$),($a$),($\lambda$) and ($v$) are established
and from ($v$) we deduce
\begin{equation}\label{v_control}
\Vert v \Vert
\lesssim
\frac{\vert \nabla K(a) \vert}{\lambda}+\frac{1}{\lambda^{2}}
\ll \epsilon
\end{equation}
as long as $u\in V(1,\epsilon)$.  And during that time 
$
\partial_{t} (K^{-1}(a)\ln(\lambda))
>0
$
by ($a$) and ($\lambda$),
since $\Delta K(a)<0$ close to $x_{0},x_{1},x_{2}^{1}$, which for energy reasons are the only reachable critical points of $K$. 
Hence $\lambda$ cannot substantially decrease and we may assume
\begin{equation}\label{lambda_control}
\sfrac{1}{\lambda} \ll \epsilon.
\end{equation}
as long as $u\in V(1,\epsilon)$. But then from \eqref{v_control} and \eqref{lambda_control} we conclude, that $u$ can actually never leave 
$V(1,\epsilon)$. Moreover from \eqref{lambda_control} and ($a$)
we then find, that 
$a$ converges to $x_{0},x_{1}$ or $x_{2}^{1}$, where by ($\lambda$) evidently 
$\lambda\longrightarrow \infty$ and due to \eqref{v_control} also
$\Vert v \Vert\longrightarrow 0$. 

\

We summarise this by saying, that full neighbourhoods of all the pure critical points at infinity 
remain concentrated along the negative gradient flow. In particular any flow line leading from one of them to the other does so in a concentrated way. 
Note, that similar shadow flow analyses of given flows have been performed in 
\cite{Bahri_Critical_Points_At_Infinity},\cite{MM_Scf},\cite{MM3},\cite{MM5}. 
\item An analogous argument based on the evolution of $\alpha,a,\lambda$ and $v$ under
$\partial_{t} u=-\nabla J_{K,\tau}(u)$
will show, that also some neighbourhoods and then also the  unstable manifolds of
$u_{\tau,x_{0}},u_{\tau,x_{1}}$ and $u_{\tau,x_{2}^{1}}$
remain in some
$V(1,\epsilon)$, i.e. remain of type $u=\alpha \delta_{a,\lambda}+v$ and of course one may choose $\epsilon\longrightarrow 0$ as $\tau\longrightarrow 0$.
This tells us, that the Morse complex, which these critical points induce, 
is completely and classically describable and this without any homological inconsistency. 
We leave this claim without a proof, since the concrete computations would go beyond the scope of this paper. However note, that transversality of the intersection
of the stable and unstable manifolds  is generic, since $u_{\tau,x_{0}}$ is a minimum
and for the relative Morse index, cf. \cite{Abbondandolo}, we have
$$
m(J_{K,\tau}u_{\tau,x_{2}^{1}})
-
m(J_{K,\tau}u_{\tau,x_{1}})
=
1
$$
\end{enumerate}
\end{enumerate}
 
Now, while neither for $K_{1}$ satisfying 1.) nor for $K_{2}$ satisfying 2.) we see any \textit{obvious} obstruction to non existence, we may ask, whether assumed non existence for both 
$K_{1}$ and $K_{2}$ is compatible, i.e. whether
$\{\partial J_{K_{1}}=0\}=\emptyset = \{\partial J_{K_{2}}=0 \}$
leads to a contradiction, hence establishing
$\{\partial J_{K_{1}}=0\}\neq \emptyset$ or $\emptyset \neq \{\partial J_{K_{2}}=0 \}$.
And this is exactly, what Theorem \ref{thm1} addresses.

\section{Appendix}
\subsection{Homotopy Equivalences}
We derive some elementary results on homotopy equivalences, denoted by  
$X\simeq Y$
for two topological sets $X$ and $Y$, i.e. the existence of continuous mappings
\begin{equation*}
\begin{split}
f:X\longrightarrow Y \; \text{ and } \; g:Y\longrightarrow X
\; \text{ with } \; 
g\circ f \simeq id_{X},\; 
f\circ g \simeq id_{Y},
\end{split}
\end{equation*}
where we denote by $\phi\simeq \varphi$ the existence of a homotopy
\begin{equation*}
\begin{split}
H: X \times [0,1]\longrightarrow Y
\; \text{ with } \; 
H(\cdot,0)=\phi
\; \text{ and } \; 
H(\cdot,1)=\varphi.
\end{split}
\end{equation*}
Moreover and in order to have a granular description we convene for
$A \subset C$ and $B \subset D$
to say, that
\begin{enumerate}[label=(\roman*)]
\item \quad $A$ is homeomorphic in $C$ to $B$ in $D$, if there exists a homeomorphism
\begin{equation*}
\begin{split}
h\in Homeo(C,D) \quad \text{ with } \quad h(A)=B
\end{split}
\end{equation*}   
\item \quad $A$ is homeomorphic to $B$ in $C=D$, if $C=D$ in the former case
\item \quad $A$ is homeomorphic to $B$ in case $A=C$ and $B=D$.   
\end{enumerate}
Finally recall, that 
$Y\subseteq X$ is called a weak deformation retract, if there exists
\begin{equation*}
\begin{split}
H: X \times [0,1] \longrightarrow X
\; \text{ with } \; 
H\lfloor_{ Y \times [0,1]}: Y \times [0,1] \longrightarrow Y
\end{split}
\end{equation*}
satisfying 
$H(\cdot,0)=id_{X}$
and
$H(\cdot,1):X\longrightarrow Y$.
A weak deformation retract induces a homotopy equivalence and, if satisfying
$H(y,t)=y$ for all $(y,t)\in  Y \times [0,1]$, 
is called a strong deformation retract. 

\begin{lemma}\label{lem_homotopy_equivalence} 
Let $A\subseteq B \subseteq C \subseteq D$ be such, that
$A\simeq C$ and $B \simeq D$
are homotopy equivalent under 
\begin{equation*}
\tilde f:D\longrightarrow B
\; \text{ and } \; 
\hat f:C\longrightarrow A
\end{equation*}
with inverse
\begin{equation*}
\tilde g:B \longrightarrow D
\; \text{ and }\; 
\hat g : A \longrightarrow C.
\end{equation*}
Then $A\simeq B \simeq C \simeq D$ are homotopically equivalent, provided there exist
\begin{enumerate}[label=(\roman*)]
\item \quad 
$\hat g \simeq \hat g^{\sharp}:A\longrightarrow C$ 
with 
$Im(\hat g^{\sharp})\subseteq B$ 
\item \quad
$\tilde g^{\sharp}:C\longrightarrow D$ 
with 
$\tilde g^{\sharp}\lfloor_{B} \; \simeq \tilde g$.  
\end{enumerate}
\end{lemma}
\begin{proof}
Let us first recall 
\begin{equation*}
\tilde g \circ \tilde f\simeq id_{D}
,\;
\tilde f \circ \tilde g \simeq id_{B}
,\; 
\hat g \circ \hat f\simeq id_{C}
,\;
\hat f \circ \hat g \simeq id_{A}
\end{equation*}
and note, that proving $A\simeq D$ is sufficient. To that end consider
\begin{equation*}
F
=
\hat f \circ \tilde f
:
D 
\overset{\tilde f}{\longrightarrow} 
B \subseteq C 
\overset{\hat f}{\longrightarrow} 
A
\end{equation*}
and 
\begin{equation*}
G
=
\tilde g^{\sharp} \circ \hat g^{\sharp}
:
A
\overset{\hat g^{\sharp}}{\longrightarrow}
B\subseteq C
\overset{\tilde g^{\sharp}}{\longrightarrow}
D.
\end{equation*}
We then have 
\begin{equation*}
G\circ F
=
\tilde g^{\sharp} \circ \hat g^{\sharp} \circ \hat f \circ \tilde f
\simeq 
\tilde g^{\sharp} \circ \hat g \circ \hat f \circ \tilde f
\simeq 
\tilde g^{\sharp} \circ id_{C} \circ \tilde f
=
\tilde g^{\sharp}  \circ \tilde f
\simeq 
\tilde g  \circ \tilde f
\simeq 
id_{D}
\end{equation*}
and 
\begin{equation*}
F\circ G
=
\hat f \circ \tilde f \circ \tilde g^{\sharp} \circ \hat g^{\sharp} 
\simeq 
\hat f \circ \tilde f \circ \tilde g \circ \hat g^{\sharp} 
\simeq 
\hat f \circ id_{B} \circ \hat g^{\sharp} 
=
\hat f \circ \hat g^{\sharp} 
\simeq 
\hat f \circ \hat g 
\simeq 
id_{A}. 
\end{equation*}
Hence $F$ with inverse $G$ provides a homotopy equivalence 
$A\simeq D$. 
\end{proof}
\begin{corollary}\label{cor_homotopy_equivalence}
Let $A\subseteq B \subseteq C \subseteq D$ be such, that
\begin{enumerate}[label=(\roman*)]
\item \quad  $A$ is homeomorphic to a weak deformation retract of $C$in $C$ 
\item \quad	$B$ is homeomorphic to a weak deformation retract of $D$ in $D$.
\end{enumerate}
Then  $A\simeq B \simeq C \simeq D$ are homotopy equivalent.
\end{corollary}
\begin{proof}
We verify the conditions of Lemma \ref{lem_homotopy_equivalence}. By assumption there exist
\begin{equation*}
\begin{split} 
\hat h\in  Homeo(C) 
\; \text{ and  } \;  
\hat f^{\prime}:C\longrightarrow \hat h(A),\;
\hat g^{\prime}:\hat h(A)\longrightarrow C
\end{split}
\end{equation*}
satisfying 
\begin{equation*}
\begin{split}
\hat g^{\prime}\circ \hat f^{\prime}\simeq id_{C}
\; \text{ and } \; 
\hat f^{\prime}\circ \hat g^{\prime }\simeq id_{\hat h(A)},
\; \text{ where } \; 
\hat g^{\prime}\simeq id_{\hat h(A)}.
\end{split}
\end{equation*}
Putting 
$\hat f = \hat h^{-}\circ \hat f^{\prime } \circ \hat h$
and
$\hat g = \hat h^{-}\circ \hat g^{\prime} \circ \hat h\lfloor_{A}$
we then have 
\begin{equation*}
\begin{split}
\hat g \circ \hat f \simeq id_{C}
\; \text{ and  } \; 
\hat f \circ \hat g\simeq id_{A},
\end{split}
\end{equation*}
i.e. $A \simeq C$ under $\hat f$, 
while 
\begin{equation*}
\begin{split}
\hat g
= 
\hat h^{-} \circ \hat g^{\prime} \circ \hat h\lfloor_{A}
\simeq 
\hat h^{-} \circ  \hat h\lfloor_{A}
=
id_{C}\lfloor_{A}
=
\hat g^{\sharp}
\; \text{ satisfies } \; 
Im(\hat g^{\sharp})=A\subseteq B.
\end{split}
\end{equation*}
Moreover by assumption there exist
\begin{equation*}
\begin{split}
\tilde h\in Homeo(D)
\; \text{ and } \; 
\tilde f^{\prime}:D\longrightarrow \tilde h(B),\;
\tilde g^{\prime}:\tilde h(B) \longrightarrow D
\end{split}
\end{equation*}
satisfying 
\begin{equation*}
\begin{split}
\tilde g^{\prime} \circ \tilde f^{\prime}\simeq id_{D}
\; \text{ and } \; 
\tilde f^{\prime} \circ \tilde g^{\prime} \simeq id_{\tilde h(B)},
\end{split}
\end{equation*}
where
$
\tilde g^{\prime} \simeq id_{\tilde h(B)}. 
$
Putting 
$\tilde f= \tilde h^{-}\circ \tilde f^{\prime}$
and
$\tilde g= \tilde g^{\prime} \circ \tilde h\lfloor_{B}$
we then have 
$\tilde g \circ \tilde f\simeq id_{D}$ 
and 
$\tilde f \circ \tilde g \simeq id_{B}$, while 
\begin{equation*}
\begin{split}
\tilde g^{\sharp}=\tilde h\lfloor_{C} \; \text{ satisfies } \;
\tilde g^{\sharp}\lfloor_{B}=\tilde h \lfloor_{B}
=
id_{\tilde h(B)}\circ \tilde h\lfloor_{B}
\simeq 
\tilde g^{\prime} \circ \tilde h\lfloor_{B}
=
\tilde g.
\end{split}
\end{equation*} 
\end{proof}

\subsection{The Sign of the Laplacian}\label{sec_sign_of_laplacian}
We consider a Morse function $K>0$ on $M$ and some critical point 
$0\simeq x_{0} \in \{\nabla K=0\}$,
i.e. $K$ takes in a suitable chart around $x_{0}$ the form
$$K(x)
=
K(x_{0})
+
b^{i}x_{i}^{2}.
$$
Clearly $\Delta K(x_{0})=2\sum_{i=1}^{n}b_{i}$. Let us consider for some $0<\delta \ll 1$ fixed
\begin{equation*}
\tilde K
=
\begin{cases}
K \; \text{ on } \; M \setminus B_{\delta}(x_{0}) \\
K(x_{0})+\sum_{i=1}^{n} \left[b_{i}\eta_{\delta}+c_{i}(1-\eta_{\delta})\right]x_{i}^{2}
\end{cases}
\end{equation*}
for some $c_{i}\in \R$ and a cut-off function 
$\eta_{\delta}=\eta(\frac{\vert x \vert}{\delta})$, where
\begin{equation*}
\eta \in C^{\infty}([0,\infty))
\; \text{ with } \; 
0\leq \eta^{\prime}\leq 2
,\; 
\eta \lfloor_{\left[ 0,1\right]}=0
\; \text{ and }\; 
\eta \lfloor_{\left[ 2,\infty\right)}=1.
\end{equation*}
Clearly 
$\tilde K(x_{0})=K(x_{0})$,
\begin{equation}\label{Change_The_Laplacian}
\Delta \tilde K(x_{0})=2\sum_{i=1}^{n}c_{i}
\end{equation}
and there holds
on $B_{2\delta}(x_{0})$
\begin{equation*}
\begin{split}
\partial_{j}\tilde K(x)
= &
2[b_{j}\eta_{\delta}(x)+c_{j}(1-\eta_{\delta}(x))]x_{j}
+ 
\delta^{-1}\sum_{i}(b_{i}-c_{i})\eta^{\prime}(\frac{\vert x \vert}{\delta})\frac{x_{j}}{\vert x \vert }x_{i}^{2}.
\end{split}
\end{equation*} 
In particular choosing 
$\forall_{1\leq j \leq n}
\vert c_{j}-b_{j}\vert \leq  \epsilon \vert b_{j}\vert$ 
for some fixed $\epsilon>0$
we find, that $\Vert K-\tilde K\Vert_{C^{0}}= O(\varepsilon \delta^{2})$ and 
\begin{equation*}
\begin{split}
\partial_{j}\tilde K(x)
= &
2b_{j}(1+O(\varepsilon))x_{j}
\end{split}
\end{equation*}
and hence $\{\nabla \tilde K=0\}=\{\nabla K=0\}$. By iteration we may therefore deform 
$K$ to any $\tilde K$ with the properties
\begin{enumerate}[label=(\roman*)]
\item 
 $\Vert \tilde K -K\Vert_{C^{0}}\ll 1$ and 
$\{\nabla \tilde K=0\}=\{\nabla K=0\}$ 
\item $\tilde K$ is a Morse function as is $K$
\item
$\tilde K=K$ on $\{ \nabla K=0 \} $
and close to any $x \in \{\nabla K=0\}$ there holds
 \begin{equation*}
W_{u}^{K}(x)=W_{u}^{\tilde K}(x)
\; \text{ and } \; W_{s}^{K}(x)=W_{s}^{\tilde K}(x)
 \end{equation*}
 for the stable and unstable manifolds $W_{u},W_{s}$ of $K,\tilde K$ of $x$.
\end{enumerate}
Note, that the value of $\Delta \tilde K$ 
at a non extremal critical point 
may be made positive or negative, cf. \eqref{Change_The_Laplacian}.


\begin{thebibliography}{flushleft}

\bibitem{Abbondandolo}
{
Abbondandolo A., Majer P.  
\textit
{  
Lectures on the Morse Complex for Infinite-Dimensional Manifolds.
} 
Morse Theoretic Methods in Nonlinear Analysis and in Symplectic Topology. NATO Science Series II: Mathematics, Physics and Chemistry, vol 217, (2006), Springer, Dordrecht
}


\bibitem{Ahmdou_Non_Simple_Blow_Up}
{
Ahmedou, M., Ben Ayed, M.
\textit{Non simple blow ups for the Nirenberg problem on half spheres.}
Discrete and Continuous Dynamical Systems, 2022, 42(12), 5967-6005
}

\bibitem{Ahmedou_Degree_Zero}
{
Ahmedou, M., Hichem Chtioui, H.
\textit
{
Conformal metrics of prescribed scalar curvature
on 4-manifolds: the degree zero case.
}
Arab. J. Math. (2017) 6:127-136
}


 
\bibitem{Bahri_Invariant}
{ 
Bahri A. 
\textit
{
An invariant for Yamabe type flows with applications
to scalar curvature problems in higher dimensions.
} 
Duke Mathematical Journal, 
81 (1996), 
323-466.
}


\bibitem{Bahri_Critical_Points_At_Infinity}
{ 
Bahri, A. 
\textit
{
Critical points at infinity in some variational problems.
}
Pitman Res. Not. in Math., 182. Longman Scientific \& Technical; copub. with John Wiley \& Sons, New York, 1989
}


\bibitem{Bahri_Coron_S_3}
{
Bahri A., Coron J-M. 
\textit{The Scalar-Curvature Problem on
the Standard Three- Dimensional Sphere.} 
Journal of Functional Analysis 
95 (1991) ,  106-172 
}




\bibitem{Ben_Ayed_M_4} 
{
Ben Ayed M., Chen Y., Chtioui H., Hammami M.
\textit{On the prescribed scalar curvature problem on 4-manifolds.}
Duke Mathematical Journal, 
84 (1996), 
633-677.
}
   
\bibitem{Bianchi_Min_Max}  
{
Bianchi G. 
\textit{
The scalar curvature equation on $\R^n$ and on $S^n$.
} 
Adv. Diff. Eq., 
1 (1996), 
857-880.
}

\bibitem{Bourguignon}
{
Bourguignon, J.P., Ezin, J.P. 
\textit
{
Scalar Curvature Functions in a Conformal Class of Metrics and Conformal Transformations.
} 
Trans. of the American Math. Soc., 301, no. 2 (1987)
}


\bibitem{Catrina_Symmetric}
{
Catrina, F.,  Wang, Z.Q. 
\textit
{Symmetric Solutions for the Prescribed Scalar Curvature Problem.}
Indiana University Mathematics Journal
Vol. 49, No. 2 ( 2000), pp. 779-813
}



\bibitem{Chang_S2_S3}
{
Chang S.A., Gursky M. J., Yang P.
\textit{The scalar curvature equation on 2- and 3-spheres.}
Calc. Var.PDE,
1 (1993),
205-229.
}

\bibitem{Chang_Conformal_Deformations} 
{
Chang S.A., Yang P.  
\textit
{
Conformal deformation of metrics on $S^2$.
} 
J.D.G., 
27 (1988), 
256-296.
}

\bibitem{Chang_S_2}
{
 Chang S.A., Yang P.  
\textit
{Prescribing Gaussian  curvature on $S^2$.} 
Acta Math., 
159 (1987), 
215-259.
}

\bibitem{Chen_Infinite_Energy_Blow_Up}
{  
Chen C.C., Lin C.S.
\textit
{
Blowing up with infinite energy of conformal metrics on $S^n$. 
}
Comm. Partial Differential Equations, 
24 (1999), 
no. 5-6, 785-799.
}



\bibitem{Escobar_Schoen_Deformation_Flatness} 
{
Escobar J., Schoen R.M. 
\textit
{
Conformal metrics with prescribed scalar curvature.
} 
Invent. Math., 
86 (1986), 
243-254.
} 



\bibitem{Kazdan_Warner_Gaussian}
{ 
Kazdan J.L., Warner F.
\textit
{ 
Existence and conformal deformation of metrics with prescribed Gaussian and scalar curvature
}
Ann. of Math., 
101 (1975), 
317-331.
}
 
\bibitem{Kazdan_Warner_Scalar}
{
Kazdan J.L., Warner F.  
\textit
{
Scalar curvature and conformal deformation of Riemannian structure.
} 
J. Differential Geometry, 
10 (1975), 
113-134.
}





\bibitem{MM1} 
{
Malchiodi A., Mayer M.
\textit{
Prescribing Morse scalar curvatures: blow-up analysis.}
Int.Mat.Res.Not., Vol. 2021, Issue 16, p. 12532-12612
rnaa021, 2020. 
}

\bibitem{MM2} 
{
Malchiodi A., Mayer M. 
\textit{
Prescribing Morse scalar curvatures: subcritical blowing-up solutions.
}
Journal of Differential Equations,
268 (2020),
no. 5, 2089-2124. 
}


\bibitem{MM4} 
{
Malchiodi A., Mayer M. 
\textit{Prescribing Morse scalar curvatures: pinching and Morse theory.}
Comm. in Pure and Appl. Math., Vol 76 (2023), Issue2, pp. 406-450
}


 
\bibitem{MM_Scf} 
{Mayer, M. 
\textit{A scalar curvature flow in low dimensions.} 
Calc. Var. Partial Differential Equations 56 (2017), no. 2, Art. 24, pp. 41 
}  



\bibitem{MM3}  
{
Mayer M.
\textit{
Prescribing Morse scalar curvatures: critical points at infinity.}
Advances in Calculus of Variations, vol. 15, no. 2, 2022, pp. 151-190
}

\bibitem{MM5} 
{
Mayer M.
\textit{Prescribing scalar curvatures: non compactness versus critical points at infinity.} 
Geometric Flows, 4(1), pp. 51-82
}

\bibitem{MM8}
{ 
Mayer, M.
\textit
{
Prescribing Morse scalar curvatures: the Gaussian case
}
To Appear
}

\bibitem{MM7}
{
Mayer M., Zhu C. 
\textit{Prescribing scalar curvatures: on the negative Yamabe case.}
preprint, \url{https://arxiv.org/abs/2302.02435}
}



 
\bibitem{Li_Part_1}
{
Li, Y. 
\textit{
Prescribing Scalar Curvature on $S^{n}$ and related problems, Part I.}
Journal of Differential Equations 120, 319-410 (1995)
}

\bibitem{Li_Part_2}
{
Li, Y. 
\textit{Prescribing scalar curvature on $S^{n}$ and related problems. II. Existence and compactness.} Comm. Pure Appl. Math. 49 (1996), no. 6, 541-597
}








\bibitem{Yacoub_Invariant_High_Dimensions}
{
Yacoub, R.
\textit{On the Scalar Curvature Equations in high dimension}
Adv. Nonl. Stud., 2 (2002), 373-393
}


\end{thebibliography}
\end{document}